\documentclass[11pt]{article}
\usepackage[latin9]{inputenc}
\usepackage{geometry}
\usepackage{color}
\usepackage{amsmath}
\usepackage{amsthm}
\usepackage{amssymb}
\usepackage{graphicx}
\usepackage[all]{xy}
\usepackage[unicode=true,pdfusetitle,
 bookmarks=true,bookmarksnumbered=false,bookmarksopen=false,
 breaklinks=false,pdfborder={0 0 0},pdfborderstyle={},backref=false,colorlinks=true]
 {hyperref}
\hypersetup{
 pdfborderstyle={},pdfborderstyle={},pdfborderstyle={},linkcolor=blue,citecolor=blue}

\makeatletter
\theoremstyle{plain}
\newtheorem{thm}{\protect\theoremname}
\theoremstyle{remark}
\newtheorem{rem}[thm]{\protect\remarkname}
\theoremstyle{plain}
\newtheorem{prop}[thm]{\protect\propositionname}
\theoremstyle{plain}
\newtheorem{cor}[thm]{\protect\corollaryname}
\theoremstyle{plain}
\newtheorem{conjecture}[thm]{\protect\conjecturename}
\theoremstyle{definition}
\newtheorem{defn}[thm]{\protect\definitionname}
\theoremstyle{plain}
\newtheorem{lem}[thm]{\protect\lemmaname}

\usepackage{amsfonts}

\newtheorem{theorem}{Theorem}[section]

\theoremstyle{definition}
\newtheorem{definition-and-remark}[theorem]{Definition and Remark}\newtheorem{notation}[theorem]{Notation}\newtheorem{notation-and-remark}[theorem]{Notation and Remark}\newtheorem{remark-and-notation}[theorem]{Remark and Notation}\newtheorem{remark-and-definition}[theorem]{Remark and Definition}

\newcommand{\cA}{ {\cal A} }
\newcommand{\bC}{ {\mathbb C} }

\newcommand{\cH}{ {\cal H} }

\newcommand{\fM}{ {\mathfrak M} }
\newcommand{\bN}{ {\mathbb N} }

\newcommand{\cS}{{\cal S}}

\newcommand{\fT}{ {\mathfrak T} }

\newcommand{\bigO}{{\cal O}}

\newcommand{\Cat}{ \mathrm{Cat} }

\newcommand{\Int}{ \mathrm{Int} }
\newcommand{\intc}{\widetilde{\Int}}
\newcommand{\Kr}{ \mathrm{Kr} }

\newcommand{\NC}{ \nc }

\newcommand{\mst}{\overline{\fM_1}}
\newcommand{\msst}{\overline{\fM}}
\newcommand{\forget}{\mathsf{forget}}

\newcommand{\dH}{ d_{ { }_H } }

\newcommand{\nc}{\mathrm{NC}}
\newcommand{\ncp}{\mathrm{NCP}}

\newcommand{\veetild}{\widetilde{\vee}}

\newcommand{\limn}{ \lim_{n \to \infty} }

\newcommand{\ecpi}{ \stackrel{\pi}{\sim} }
\newcommand{\ecrho}{ \stackrel{\rho}{\sim} }

\newcommand{\la}{\lambda}

\newcommand{\bw}{\mathbf{w}}
\newcommand{\bz}{\mathbf{z}}
\newcommand{\bi}{\mathbf{i}}
\newcommand{\bj}{\mathbf{j}}

\providecommand{\corollaryname}{Corollary}
\providecommand{\theoremname}{Theorem}

\reversemarginpar
\numberwithin{thm}{section}
\numberwithin{equation}{section}

\usepackage[perpage,symbol]{footmisc}
\usepackage[leftcaption]{sidecap}

\sidecaptionvpos{figure}{c}

\newcommand{\FigBesBeg}[1][1.0]{%
 \let\MyFigure\figure
 \let\MyEndfigure\endfigure
 \renewenvironment{figure}[1]{\begin{SCfigure}[#1]##1}{\end{SCfigure}}}

\newcommand{\FigBesEnd}{%
 \let\figure\MyFigure
 \let\endfigure\MyEndfigure}

  \providecommand{\conjecturename}{Conjecture}
  \providecommand{\corollaryname}{Corollary}
  \providecommand{\definitionname}{Definition}
  \providecommand{\lemmaname}{Lemma}
  \providecommand{\propositionname}{Proposition}
  \providecommand{\remarkname}{Remark}
\providecommand{\theoremname}{Theorem}

  \providecommand{\conjecturename}{Conjecture}
  \providecommand{\corollaryname}{Corollary}
  \providecommand{\definitionname}{Definition}
  \providecommand{\lemmaname}{Lemma}
  \providecommand{\propositionname}{Proposition}
  \providecommand{\remarkname}{Remark}
\providecommand{\theoremname}{Theorem}

\makeatother

\providecommand{\conjecturename}{Conjecture}
\providecommand{\corollaryname}{Corollary}
\providecommand{\definitionname}{Definition}
\providecommand{\lemmaname}{Lemma}
\providecommand{\propositionname}{Proposition}
\providecommand{\remarkname}{Remark}
\providecommand{\theoremname}{Theorem}

\begin{document}
\title{\textbf{\Large{}Asymptotics for a Class of Meandric Systems,}\\
\textbf{\Large{}via the Hasse Diagram of NC(n)}}
\author{I. P. Goulden\thanks{Research supported by a Discovery Grant from NSERC, Canada},~~
Alexandru Nica\thanks{Research supported by a Discovery Grant from NSERC, Canada},~~
Doron Puder\thanks{Research supported by the Israeli Science Foundation (grant No.~1071/16)}}
\maketitle
\begin{abstract}
We consider the framework of closed meandric systems, and its equivalent
description in terms of the Hasse diagrams of the lattices of non-crossing
partitions $\NC(n)$. In this equivalent description, considerations
on the number of components of a random meandric system of order $n$
translate into considerations about the distance between two random
partitions in $\NC(n)$. We put into evidence a class of couples $(\pi,\rho)\in\nc(n)^{2}$
-- namely the ones where $\pi$ is conditioned to be an interval
partition -- for which it turns out to be tractable to study distances
in the Hasse diagram. As a consequence, we observe a non-trivial class
of meanders (i.e.~connected meandric systems), which we call ``meanders
with shallow top'', and which can be explicitly enumerated. Moreover,
denoting by $c_{n}$ the expected number of components for the corresponding
notion of ``meandric system with shallow top'' of order $n$, we
find the precise asymptotic $c_{n}\approx\frac{n}{3}+\frac{28}{27}$
for $n\to\infty$. Our calculations concerning expected number of
components are related to the idea of taking the derivative at $t=1$
in a semigroup for the operation $\boxplus$ of free probability (but
the underlying considerations are presented in a self-contained way,
and can be followed without assuming a free probability background).

Let $c_{n}'$ denote the expected number of components of a general,
unconditioned, meandric system of order $n$. A variation of the methods
used in the shallow-top case allows us to prove that $\mathrm{lim\ inf}_{n\to\infty}c_{n}'/n\geq0.17$.
We also note that, by a direct elementary argument, one has $\mathrm{lim\ sup}_{n\to\infty}c_{n}'/n\leq0.5$.
These bounds support the conjecture that $c_{n}'$ follows a regime
of ``constant times $n$'' (where numerical experiments suggest
that the constant should be $\approx0.23$). 
\end{abstract}
\tableofcontents{}

\section{{\large{}Introduction\label{sec:Introduction}}}

\subsubsection*{Framework of the paper}

In this paper we consider the framework of closed meandric systems,
and its equivalent description in terms of Hasse diagrams (that is,
graphs of covers) of the lattices of non-crossing partitions $\NC(n)$.
In this equivalent description, considerations on the number of components
of a random meandric system of order $n$ translate into considerations
about the distance between two random partitions $\pi,\rho\in\NC(n)$,
where $\pi,\rho$ are viewed as vertices in the Hasse diagram. The
relevant definitions and statements of facts concerning these frameworks
will be reviewed precisely in Sections \ref{sec:Background-on-NC(n)}
and \ref{sec:Meandric-systems} below. Here we only record the following
important point: denoting the set of closed meandric systems of order
$n$ by \marginpar{$\fM\left(n\right)$}$\fM(n)$, one has a natural
bijection $\NC(n)^{2}\ni(\pi,\rho)\mapsto M(\pi,\rho)\in\fM(n)$,
with the property that 
\begin{equation}
\left(\begin{array}{c}
\mbox{number of components}\\
\mbox{of \ensuremath{M(\pi,\rho)}}
\end{array}\right)\ =\ n-\dH(\pi,\rho),\ \ \forall\,\pi,\rho\in\nc\left(n\right),\label{eq:M=00003Dn-d_H}
\end{equation}
where $\dH$ is the distance function in the Hasse diagram of $\nc\left(n\right)$.
The picture of the meandric system $M(\pi,\rho)$ is obtained by ``doubling''
the partitions $\pi$ and $\rho$, and by drawing the resulting doublings
above and respectively below a horizontal line with $2n$ points marked
on it. (See Section \ref{sec:Meandric-systems} for a precise description
of this.)

When looking at meandric systems, two appealing questions that immediately
arise are: 
\begin{itemize}
\item What is the average number of components of a random meandric system
in $\fM(n)$? 
\item What is the probability that a random meandric system in $\fM(n)$
is connected\linebreak{}
 (that is, it has only one component)? 
\end{itemize}
\noindent The latter question is a rather celebrated one: a connected
meandric system goes under the short name of \emph{meander}, and since
the cardinality of $\fM(n)$ is easily calculated as the square Catalan
number $\Cat_{n}^{2}$, this question actually asks what is the number
of (closed) meanders of order $n$. This is known to be a hard problem,
which is open even at the level of finding the asymptotic growth rate
as $n\to\infty$ (see \cite{di2000folding,lando2004graphs}). The
question about average number of components seems to have received
less attention in the literature, but is very natural; we suspect
it to be rather hard as well. As explained below, there are reasons
to believe that the average mentioned in this question must behave
like $c\cdot n$ for a constant $c\in(0.17,0.50)$ (which, if exists,
is yet to be determined).

When converted to the framework of the Hasse diagram of $\NC(n)$,
the two questions stated above become: 
\begin{itemize}
\item What is the average distance between two random partitions in $\NC(n)$? 
\item What is the probability that two random partitions $\pi,\rho\in\NC(n)$
have $\dH(\pi,\rho)=n-1$\linebreak{}
 (i.e.~that $\pi$ and $\rho$ ``achieve a diameter'' in the Hasse
diagram of $\NC(n)$)? 
\end{itemize}
\noindent Since the conversion from meandric systems to the Hasse
diagrams of $\NC(n)$ is rather straightforward, the new questions
are about as tractable (or intractable) as the original ones.

\subsubsection*{Meandric systems with shallow top}

In this paper we mainly study distances between $\pi,\rho\in NC(n)$
where $\pi$ is conditioned to be an \emph{interval} partition --
that is, every block of $\pi$ is of the form $\{i\in\bN\mid p\leq i\leq q\}$
for some $p\leq q$ in $\{1,\ldots,n\}$. We denote the set of interval
partitions of order $n$ by $\mathrm{Int}\left(n\right)$\marginpar{$\mathrm{Int}\left(n\right)$}.
As motivation for why one might look at interval partitions, we can
invoke the state of mind which comes from using partitions in the
combinatorial development of non-commutative probability. When looking
for non-commutative analogs for classical probability theory, a case
can be made (cf. \cite{speicher1997}) that there are two main analogs
which can be pursued: free probability, whose combinatorics is based
on the lattices $\NC(n)$, and the lighter theory of Boolean probability,
whose combinatorics is based on the lattices $\Int(n)$ of interval
partitions. There is no doubt that meandric systems have connections
to free probability, and the specialization ``$\pi\in\Int(n)$''
can be construed as ``going Boolean in the first variable''. The
cost of this specialization is that the cardinality of $\Int(n)$
is roughly only the square root of that of $\NC(n)$: $|\Int(n)|=2^{n-1}$,
while $|\NC(n)|=\Cat_{n}$, with magnitude of roughly $4^{n}$. But
on the bright side, one has the precise results stated in Theorems
\ref{thm:main-shallow meanders} and \ref{thm:avg distance from interval partition}
below.

We denote by $\msst\left(n\right)$ the set of meandric systems of
order $n$ corresponding to $\left(\pi,\rho\right)\in\Int\left(n\right)\times\nc\left(n\right)$.
In the statement of Theorem \ref{thm:main-shallow meanders}, the
``1'' in the notation $\mst\left(n\right)$ marks the fact that
a meandric system corresponding to a couple $(\pi,\rho)\in\mst\left(n\right)$
has $1$ component. We are also using the notation $\left|\pi\right|$\marginpar{$\left|\pi\right|$}
for the number of blocks in a partition $\pi\in NC(n)$. 
\begin{thm}
\label{thm:main-shallow meanders} For every $n\in\bN$, denote\marginpar{$\mst\left(n\right)$}
\begin{equation}
\mst\left(n\right)=\{(\pi,\rho)\in\Int(n)\times\NC(n)\mid\dH(\pi,\rho)=n-1\}.\label{eqn:11a}
\end{equation}
Then one has 
\begin{equation}
\left|\mst\left(n\right)\right|=\sum_{m=1}^{n}\frac{1}{n}\left(\begin{array}{c}
n\\
m-1
\end{array}\right)\,\left(\begin{array}{c}
n+m-1\\
n-m
\end{array}\right).\label{eqn:meanders with shallow top}
\end{equation}
Moreover, in the latter sum, the term indexed by every $m\in\{1,\ldots,n\}$
gives precisely the number of couples $(\pi,\rho)$ counted on the
right-hand side of (\ref{eqn:11a}) and where $\left|\pi\right|=m$. 
\end{thm}

\begin{rem}
\label{rem:exp growth rate of shallow meanders} 
\begin{enumerate}
\item The statement of Theorem \ref{thm:main-shallow meanders} follows
from a more general formula, given in Proposition \ref{prop:N(m,n,j)}
below, which in turn is a consequence of a fairly subtle tree bijection
found in Section 6 of the paper. Some other specializations of Proposition
\ref{prop:N(m,n,j)} are discussed at the end of this Introduction
(see Proposition \ref{prop:16} and the comment that follows it).
\item Let us say that $\rho\in\nc\left(n\right)$ is a ``meandric partner''
of $\pi\in\nc\left(n\right)$ if $M\left(\pi,\rho\right)$ is a meander,
or, equivalently, if $d_{H}\left(\pi,\rho\right)=n-1$. The exponential
growth rate of the meanders counted in Theorem \ref{thm:main-shallow meanders}
is roughly $5.22$, namely, $\overline{\lim}_{n\to\infty}\sqrt[n]{\left|\mst\left(n\right)\right|}\approx5.22$
(see Corollary \ref{cor:EGR of meanders with shallow top}). Since
there are $2^{n-1}$ interval partitions, this means that the average
number of meandric partners of an interval partition is roughly $2.61^{n}$.
This is less than the average for all non-crossing partitions, which
has exponential growth rate of at least 2.845 \cite[Theorem 1.1]{albert2005bounds}
and conjectured to be roughly $3.066$ \cite{jensen2000critical}.
In Conjecture \ref{conj:rainbow partitions} below we make a conjecture
as for the type of non-crossing partitions with maximal number of
meandric partners. 
\end{enumerate}
\end{rem}

\begin{thm}
\label{thm:avg distance from interval partition} For every $n\in\bN$,
consider the average distance\marginpar{$d_{n}$} 
\begin{equation}
d_{n}=\frac{1}{2^{n-1}\cdot\Cat_{n}}\sum_{\begin{array}{c}
{\scriptstyle \pi\in\Int(n),}\\
{\scriptstyle \rho\in NC(n)}
\end{array}}\ \dH(\pi,\rho).\label{eqn:12a}
\end{equation}
Then one has 
\begin{equation}
\limn\bigl(d_{n}-\frac{2}{3}n\bigr)=-28/27.\label{eqn:12b}
\end{equation}
\end{thm}

We mention that the limit stated in Equation (\ref{eqn:12b}) is found
by making $n\to\infty$ in an explicit expression which holds for
a fixed value of $n$, and which is indicated precisely in Proposition
\ref{prop:expression for d_n} below. The calculations leading to
this expression are related to the idea of taking the derivative at
$t=1$ in a semigroup for the operation $\boxplus$ of free probability.
The connection to the operation $\boxplus$ is explained in Remarks
\ref{rem:82} and \ref{rem:84} of Section 8 (but the proof of Proposition
\ref{prop:expression for d_n} is then presented in a self-contained
way, and can be followed without assuming free probability background).

It is remarkable that the difference $d_{n}-\frac{2}{3}n$ has a plain
limit as $n\to\infty$, rather than having some asymptotic behaviour
which involves lower powers of $n$. We note that this kind of phenomenon
also appears in other calculations of averages of distances in $NC(n)$,
and is reflected in the corresponding results about average number
of components in a random meandric system (see Equation (\ref{eqn:17b})
of Proposition \ref{prop:17}, and the corresponding Equations (\ref{eqn:12c})
and (\ref{eqn:15c}) for $c$'s instead of $d$'s).

A suggestive name for the meandric systems of order $n$ that correspond
to couples $(\pi,\rho)\in\Int(n)\times\NC(n)$ is ``meandric systems
with shallow top''. Indeed, the standard drawing of the meander $M\left(\pi,\rho\right)$
with a horizontal infinite line $L$, is shallow, in the sense that
every point on $L$ has at most two lines above it, if and only if
$\pi$ is an interval partition (see Sections \ref{sec:Meandric-systems}
and \ref{sec:Interval-Partitions}). When we convert back from Hasse
diagrams to meandric systems, Theorem \ref{thm:main-shallow meanders}
will thus give us the number of meanders (i.e.~connected meandric
systems) of order $n$ and with shallow top. On the other hand, denoting
by \marginpar{$c_{n}$}$c_{n}$ the expected number of components
of a random meandric system of order $n$ with shallow top, we have
$c_{n}=n-d_{n}$ by \eqref{eq:M=00003Dn-d_H}, hence Theorem \ref{thm:avg distance from interval partition}
gets converted into the statement that 
\begin{equation}
\limn\bigl(c_{n}-\frac{n}{3}\bigr)=28/27.\label{eqn:12c}
\end{equation}

\subsubsection*{Bounds for a general meandric system}

In connection to Theorem \ref{thm:avg distance from interval partition},
it is relevant to ask: if we were to average the distances $\dH(\pi,\rho)$
for all $\pi,\rho\in NC(n)$ (rather than conditioning $\pi$ to be
in $\Int(n)$), would that average still follow a regime of ``constant
times $n$'', where the constant is contained in $(0,1)$? Or, at
the very least, does that average admit some lower and upper bounds
of the form $\alpha\,n$ and $\beta\,n$, with $\alpha,\beta\in(0,1)$?
The existence of a lower bound $\alpha\,n$ is in fact immediate,
because it is easy to prove (c.f. Corollary \ref{cor:lower-bound for average dist})
that 
\[
\frac{1}{\Cat_{n}^{2}}\ \sum_{\pi,\rho\in NC(n)}\dH(\pi,\rho)\geq\frac{n-1}{2},\ \ \forall\,n\in\bN.
\]
The methods of the present paper allow us to also prove the existence
of an upper bound $\beta\,n$. In order to obtain it, we use the following
inequality: 
\begin{equation}
\dH(\pi,\rho)\leq|\pi|+|\rho|-2|\pi\vee\rho|,\ \ \forall\,\pi,\rho\in\NC(n),\label{eqn:13a}
\end{equation}
where ``$\vee$'' is the join operation in $\NC(n)$. This is just
a triangle inequality, where the right-hand side is $\dH(\pi,\pi\vee\rho)+\dH(\pi\vee\rho,\rho)$;
but it is nevertheless very useful, due to the following proposition,
which is obtained along similar lines to the proof of Theorem \ref{thm:avg distance from interval partition}
-- details are given in Section \ref{sec:Distance-distributions}
below (cf. Proposition \ref{prop:exact expression for mu_n} there). 
\begin{prop}
\label{prop:upper bound for average distance} 
\[
\lim_{n\to\infty}\frac{1}{n}\cdot\frac{1}{\Cat_{n}^{2}}\sum_{\pi,\rho\in\NC(n)}|\pi|+|\rho|-2|\pi\vee\rho|=\frac{3\pi-8}{8-2\pi}<0.83.
\]
\end{prop}

By using Proposition \ref{prop:upper bound for average distance}
and inequality \eqref{eqn:13a}, one immediately obtains the needed
upper bound for average distance $\dH(\pi,\rho)$, and hence the following
corollary.
\begin{cor}
\label{cor:14} (1) Denote by $d'_{n}$\marginpar{$d'_{n}$} the expected
distance between two randomly chosen non-crossing partitions in $\nc\left(n\right)$.
Then for any $\varepsilon>0$ and large enough $n$ 
\[
\left(0.5-\varepsilon\right)n\leq d'_{n}\leq0.83n.
\]
(2) Denote by $c_{n}'$\marginpar{$c_{n}'$} the expected number of
components of a random meandric system of order $n$. Then for any
$\varepsilon>0$ and large enough $n$ 
\[
0.17n\leq c_{n}'\leq\left(0.5+\varepsilon\right)n.
\]
\end{cor}

A natural question prompted by Corollary \ref{cor:14} is, of course:
\[
\mbox{Does}\ \lim_{n\to\infty}\frac{c_{n}'}{n}\ \mbox{ exist? (If yes, what is it?) }
\]
Computational experiments done for some fairly large values of $n$
\cite{summer-project} suggest that the limit exists and is $\approx0.23$.
This is smaller than the constants of $\approx0.33$ and $\approx0.29$
arising out of Equations \eqref{eqn:12c} and respectively \eqref{eqn:15c},
which suggests the idea that the number of components of a random
meandric system $M(\pi,\rho)$ increases when one conditions $\pi$
to be an interval partition. This agrees in spirit with the fact,
stated in Remark \ref{rem:exp growth rate of shallow meanders}, that
interval partitions have, on average, fewer meandric partners than
general non-crossing partitions.

\subsubsection*{Distances from a fixed base-point in $\Int(n)$}

A variation of the framework used in Theorems \ref{thm:main-shallow meanders}
and \ref{thm:avg distance from interval partition} is obtained by
fixing a ``base-point'' $\lambda_{n}\in\Int(n)$, and by focusing
on distances in the Hasse diagram of $\NC(n)$ which are measured
from $\lambda_{n}$. When converted to the language of meandric systems,
this corresponds to a type of question that was analyzed quite thoroughly
in Section 6.3 of the paper \cite{di1997meander}. We would like nevertheless
to discuss here a couple of illustrative examples of base-points $\lambda_{n}$,
and record some slight improvements in the formulas related to them,
which follow from the methods used in the present paper.

For the first example of this subsection, recall from \eqref{eq:M=00003Dn-d_H}
and Remark \ref{rem:exp growth rate of shallow meanders} that a partition
$\rho\in\nc\left(n\right)$ is called a meandric partner for our base-point
$\lambda_{n}$ when it satisfies $\dH(\lambda_{n},\rho)=n-1$.
\begin{prop}
\label{prop:16} For $\ell,m\in\mathbb{N}$, let $\lambda_{\ell,m}\in\Int\left(\ell m\right)$
denote the interval partition consisting of $m$ blocks of size $\ell$
each: 
\begin{equation}
\lambda_{\ell,m}=\bigl\{\,\{1,\ldots,\ell\},\{\ell+1,\ldots,2\ell\},\ldots,\{(m-1)\ell+1,\ldots,m\ell\}\,\bigr\}\in\Int(\ell m).\label{eqn:16a}
\end{equation}
Then the number of meandric partners for $\lambda_{\ell,m}$ is equal
to $\mathrm{FCat}_{m}^{(\ell)}\cdot\ell^{m-1}$, where 
\[
\mathrm{FCat}_{m}^{(\ell)}:=\frac{1}{(\ell-1)m+1}\left(\begin{array}{c}
\ell m\\
m
\end{array}\right)\ \ \mbox{ (a Fuss-Catalan number).}
\]
\end{prop}

The case $\ell=2$ of Proposition \ref{prop:16} (which says that
the number of meandric partners for $\bigl\{\,\{1,2\},\{3,4\},\ldots,\{2m-1,2m\}\,\bigr\}$
is equal to $2^{m-1}\Cat_{m}$) was obtained in Equation (6.58) of\linebreak{}
\cite{di1997meander}. The asymptotics with general fixed $\ell$
and $m\to\infty$ for the number of meandric partners of the $\lambda_{\ell,m}$
indicated in (\ref{eqn:16a}) (but not the enumeration by Fuss-Catalan
numbers) appears in Equation (6.68) of the same paper \cite{di1997meander}.
Similarly to Theorem \ref{thm:main-shallow meanders}, the statement
of Proposition \ref{prop:16} is a consequence of a more general formula,
given in Proposition \ref{prop:N(m,n,j)} below, which in turn follows
from the tree bijection presented in Section 6 of the paper.

We find it likely that the methods used in the present paper will
work for various other special situations where the $\lambda_{n}$'s
are interval partitions, but that new ideas are needed when one picks
$\lambda_{n}$'s in $\NC(n)$ which have nestings. The hardest to
handle, along these lines, is the example where $\lambda_{n}$ is
the ``rainbow'' partition of $\{1,\ldots,n\}$: this is the partition
with blocks $\{1,n\},\{2,n-1\},\ldots$ including a possible singleton
block at $(n+1)/2$ when $n$ is odd. In relation to this, we note
that numerical experiments support the idea, also brought up in the
last paragraph of Section 6 of \cite{di1997meander}, that for every
$n\in\bN$, the rainbow partition of $\{1,\ldots,n\}$ is the partition
in $\NC(n)$ which has the largest number of meandric partners. More
precisely, we make the following conjecture.
\begin{conjecture}
\noindent \label{conj:rainbow partitions} Let $n$ be a positive
integer, and consider the orbit of the rainbow partition of $\{1,\ldots,n\}$
under the Kreweras complementation map (see Section \ref{sec:Background-on-NC(n)}).
By symmetry, all $n$ partitions in this orbit have the same number
of meandric partners. We conjecture that this orbit constitutes the
exact set of partitions in $\nc\left(n\right)$ with the largest number
of meandric partners. 
\end{conjecture}

We also note here that counting the meandric partners for the rainbow
partition amounts to counting some diagrams known in the literature
under the name of \emph{semi-meanders} (see e.g.~\cite[Section 2.2]{di1997meander}
or \cite[Section 2.2.1]{lacroix2003approaches}). This is a well-known
problem, which is believed to be hard. The conjecture stated above
is not about precisely counting semi-meanders, but about proving an
inequality between the number of semi-meanders and the cardinality
of other sets of meanders with prescribed top.

Finally, let us also have a look at how things go when, instead of
focusing on meandric partners, we are interested in expected number
of components. We give, for illustration, the following proposition.
\begin{prop}
\label{prop:17} As above, let $\lambda_{2,m}=\left\{ \left\{ 1,2\right\} ,\left\{ 3,4\right\} ,\ldots,\left\{ 2m-1,2m\right\} \right\} \in\nc\left(2m\right)$,
and consider the average distance \marginpar{$\widetilde{d}_{2,m}$}
\[
\widetilde{d}_{2,m}:=\frac{1}{\Cat_{2m}}\sum_{\rho\in\nc\left(2m\right)}\,\dH\left(\lambda_{2,m},\rho\right).
\]
Then one has the explicit formula 
\begin{equation}
\widetilde{d}_{2,m}=\frac{2^{2m-1}\cdot\binom{2m}{m}}{\Cat_{2m}}-\frac{3}{2}.\label{eqn:17a}
\end{equation}
As a consequence of this formula, it follows that: 
\begin{equation}
\lim_{m\to\infty}\left(\widetilde{d}_{2,m}-\sqrt{2}m\right)=\frac{7\sqrt{2}}{16}-\frac{3}{2}.\label{eqn:17b}
\end{equation}
\end{prop}

The formulas found in Proposition \ref{prop:17} admit obvious conversions
into formulas about the meandric systems of order $2m$ which have
the top part consisting of a repetition of the double arch pattern
``$\Cap$'' arising from the doubling of $\lambda_{2,m}$. In particular,
denoting by $\widetilde{c}_{2,m}$\marginpar{$\widetilde{c}_{2,m}$}
the expected number of components of such a meandric system, one finds
that 
\begin{equation}
\lim_{m\to\infty}\left(\widetilde{c}_{2,m}-\left(1-\frac{\sqrt{2}}{2}\right)2m\right)=\frac{3}{2}-\frac{7\sqrt{2}}{16}.\label{eqn:15c}
\end{equation}
We note that the asymptotics $\widetilde{c}_{2,m}\sim(1-\frac{1}{\sqrt{2}})2m$
was also found in Equation (6.54) of \cite{di1997meander}, but without
the precise limit stated on the right-hand side of our Equation (\ref{eqn:15c}).

\subsubsection*{Paper organization}

We conclude this introduction by explaining how the paper is organized,
and by giving a few highlights on the content of the various sections.
In Sections \ref{sec:Background-on-NC(n)}, \ref{sec:Meandric-systems}
and \ref{sec:formulas-for-d_H} we introduce the framework used in
the paper, and we discuss some necessary background. Specifically,
Section \ref{sec:Background-on-NC(n)} reviews basic facts about $\left(\nc\left(n\right),\leq\right)$
and its Hasse diagram $\cH_{n}$; Section \ref{sec:Meandric-systems}
discusses the connection between these Hasse diagrams and meandric
systems; and Section \ref{sec:formulas-for-d_H} presents several
equivalent formulas for distances in $\cH_{n}$.

In Section \ref{sec:Interval-Partitions} we discuss interval partitions
and meanders with shallow top, and in Section \ref{sec:bijection},
relying on the facts from Sections \ref{sec:formulas-for-d_H} and
\ref{sec:Interval-Partitions}, we prove there is a bijection between
the set of meanders with shallow top and a certain set of finite trees.
This leads to explicit enumerative results, which are presented separately
in Section \ref{sec:Enumerative-consequences}, and include in particular
Theorem \ref{thm:main-shallow meanders} and Proposition \ref{prop:16}.
Finally, Section \ref{sec:Distance-distributions} discusses averages
of distances and presents the proofs of Theorem \ref{thm:avg distance from interval partition}
and of Propositions \ref{prop:upper bound for average distance} and
\ref{prop:17}.

\section{{\large{}Background on \boldmath{$\left(\nc\left(n\right),\protect\leq\right)$}
and Its Hasse Diagram\label{sec:Background-on-NC(n)}}}
\begin{defn}
\label{def:nc} Let $n$ be a positive integer. 
\begin{enumerate}
\item We will work with partitions of the set $\left\{ 1,\ldots,n\right\} $.
Our typical notation for such a partition is $\pi=\left\{ V_{1},\ldots,V_{k}\right\} $,
where $V_{1},\ldots,V_{k}$ (the \emph{blocks} of $\pi$) are non-empty,
pairwise disjoint sets with $\bigcup_{i=1}^{k}V_{i}=\left\{ 1,\ldots,n\right\} $.
We will occasionally use the notation ``$V\in\pi$'' to mean that
$V$ is one of the blocks of the partition $\pi$. The number of blocks
of $\pi$ is denoted by $\left|\pi\right|$.
\item We say that a partition $\pi$ of $\left\{ 1,\ldots,n\right\} $ is
\emph{non-crossing} when it is not possible to find two distinct blocks
$V,W\in\pi$ and numbers $a<b<c<d$ in $\left\{ 1,\ldots,n\right\} $
such that $a,c\in V$ and $b,d\in W$. Equivalently, $\pi$ is non-crossing
if it can be depicted in a diagram with $n$ vertices arranged on
an invisible horizontal line, so that the blocks of $\pi$ are the
connected components of a planar diagram drawn in the upper-half plane:
see Figure \ref{fig:example in NC(9)} and also e.g.~\cite[Lecture 9]{nica-speicher2006lectures}.
The set of all non-crossing partitions of $\left\{ 1,\ldots,n\right\} $
is denoted by $\nc\left(n\right)$. This is one of the many combinatorial
structures counted by Catalan numbers, namely 
\[
\left|\NC(n)\right|=\Cat_{n}:=\frac{(2n)!}{n!(n+1)!}\ \ \mbox{ (the \ensuremath{n}th Catalan number)}
\]
(see e.g.~\cite[Proposition 9.4]{nica-speicher2006lectures} and
also Remark \ref{rem:|NC(n)|=00003Dcat  and def of kr} below).
\item On $\nc\left(n\right)$ we will use the partial order given by \emph{reverse
refinement}: for $\pi,\rho$ we put 
\begin{equation}
\Bigl(\pi\leq\rho\Bigr)\ \Leftrightarrow\ \left(\begin{array}{c}
\mbox{for every \ensuremath{V\in\pi} there}\\
\mbox{exists \ensuremath{W\in\rho} such that \ensuremath{V\subseteq W}}
\end{array}\right).\label{eqn:21z}
\end{equation}
\item We denote by $0_{n}$ the partition of $\left\{ 1,\ldots,n\right\} $
into $n$ blocks of one element each, and we denote by $1_{n}$ the
partition of $\left\{ 1,\ldots,n\right\} $ into one block of $n$
elements. These are the minimum and the maximum elements in $\left(\nc\left(n\right),\le\right)$,
that is, one has $0_{n}\leq\pi\leq1_{n}$ for every $\pi\in\nc\left(n\right)$.
\item The partially ordered set $\left(\nc\left(n\right),\le\right)$ turns
out to be a \emph{lattice}, which means that every $\pi,\rho\in\nc\left(n\right)$
have a least common upper bound $\pi\vee\rho$ and a greatest common
lower bound $\pi\wedge\rho$. The partitions $\pi\vee\rho$ and $\pi\wedge\rho$
are called the \emph{join} and \emph{meet}, respectively, of $\pi$
and $\rho$. It is easily verified that $\pi\wedge\rho$ can be explicitly
described as 
\begin{equation}
\pi\wedge\rho:=\left\{ V\cap W\,\middle|\,V\in\pi,W\in\rho,V\cap W\neq\emptyset\right\} ,\label{eqn:22a}
\end{equation}
but there is no such simple explicit formula for $\pi\vee\rho$. (For
the proof that $\pi\vee\rho$ does indeed exist, see e.g.~\cite[Proposition 3.3.1]{stanley2011enumerative}.) 
\end{enumerate}
\end{defn}

\begin{figure}
\includegraphics[viewport=-60bp 300bp 100bp 350bp]{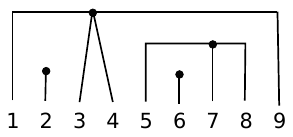}\caption{\label{fig:example in NC(9)}A non-crossing geometric realization
of the non-crossing partition $\left\{ \left\{ 1,3,4,9\right\} ,\left\{ 2\right\} ,\left\{ 5,7,8\right\} ,\left\{ 6\right\} \right\} \in\nc\left(9\right)$}
\end{figure}

\begin{defn}
\label{def:H_n}Let $n$ be a positive integer. 
\begin{enumerate}
\item Let $\pi,\rho$ be in $\nc\left(n\right)$. We will say that $\rho$
\emph{covers} $\pi$ to mean that $\pi\leq\rho$, $\pi\neq\rho$,
and there exists no $\theta\in\nc\left(n\right)\setminus\left\{ \pi,\rho\right\} $
such that $\pi\leq\theta\leq\rho$. 
\item The Hasse diagram of $\left(\nc\left(n\right),\le\right)$ is the
undirected graph \marginpar{$\cH_{n}$}$\cH_{n}$ described as follows:
the vertex set of $\cH_{n}$ is $\nc\left(n\right)$, and the edge
set of $\cH_{n}$ consists of subsets $\{\pi_{1},\pi_{2}\}\subseteq\nc\left(n\right)$
where one of $\pi_{1},\pi_{2}$ covers the other. This is illustrated
in Figure \ref{fig:H4}. 
\end{enumerate}
\end{defn}

\begin{figure}[h]
\begin{minipage}[t]{0.4\columnwidth}%
\[
\xymatrix{ &  &  &  & 1234\ar@{-}[dll]\ar@{-}[dl]\ar@{-}[d]\\
~ & ~ & 12,34\ar@{-}[d]\ar@{-}[drrrrr] & 123,4\ar@{-}[d] & 124,3\ar@{-}[dll]\ar@{-}[d]\ar@{-}[drr] & 134,2\ar@{-}[drr]\ar@{-}[ul] & 14,23\ar@{-}[ull]\ar@{-}[dl] & 1,234\ar@{-}[ulll]\ar@{-}[dll] &  &  & ~\\
 &  & 12,3,4\ar@{-}[drr]\ar@{-}[ur] & 13,2,4\ar@{-}[urr] & 14,2,3\ar@{-}[urr]\ar@{-}[ur] & 1,23,4\ar@{-}[ull]\ar@{-}[dl] & 1,24,3\ar@{-}[dll]\ar@{-}[ur] & 1,2,34\ar@{-}[dlll]\ar@{-}[u] &  &  & \hspace{1em}\\
 &  &  &  & 1,2,3,4\ar@{-}[ull]\ar@{-}[ul]\ar@{-}[u]
}
\]
\end{minipage}\caption{\label{fig:H4}The Hasse diagram $\cH_{4}$ of the lattice $\nc\left(4\right)$.
Note that all partitions of $\left\{ 1,2,3,4\right\} $ are non-crossing
except for $\left\{ \left\{ 1,3\right\} ,\left\{ 2,4\right\} \right\} $.}
\end{figure}

\begin{rem}
\label{rem:covers}It is not hard to see that for $\pi,\rho\in\nc\left(n\right)$,
one has 
\[
\Bigl(\,\mbox{\ensuremath{\rho} covers \ensuremath{\pi}}\,\Bigr)\ \Leftrightarrow\ \Bigl(\,\mbox{\ensuremath{\rho\geq\pi} and \ensuremath{|\rho|=|\pi|-1}}\,\Bigr).
\]
So if one draws the vertices of $\cH_{n}$ (that is, the partitions
in $NC(n)$) arranged on horizontal levels according to number of
blocks, then every edge of $\cH_{n}$ will connect two vertices situated
on adjacent levels -- see Figure \ref{fig:H4}. More precisely, an
edge connects a partition $\pi$ on level $k$ to a partition $\rho$
on level $k+1$ precisely when $\pi$ can be obtained from $\rho$
by taking a block $W\in\rho$ and breaking it into two pieces, in
a non-crossing way. 
\end{rem}

As announced in the Introduction, the main object of concern for the
present paper is the structure of \emph{distances} in the Hasse diagram
$\cH_{n}$.

\begin{notation} \label{not:d_H}For every positive integer $n$
and for every $\pi,\rho\in\nc\left(n\right)$ we will use the notation
\marginpar{$d_{H}\left(\pi,\rho\right)$}$\dH(\pi,\rho)$ for the
distance between $\pi$ and $\rho$ in the graph $\cH_{n}$. That
is, $\dH(\pi,\rho)$ is the length of the shortest path in $\cH_{n}$
which connects $\pi$ with $\rho$. \end{notation}

\vspace{6pt}

The next proposition records some easy observations about $\dH$: 
\begin{prop}
\label{prop:diameter of H_n}Let $n$ be a positive integer. 
\begin{enumerate}
\item The diameter of $\cH_{n}$ is $n-1$. 
\item If $\pi,\rho\in\nc\left(n\right)$ are such that $\pi\leq\rho$, then
$\dH(\pi,\rho)=|\pi|\,-\,|\rho|$. 
\item If $\pi,\rho\in\nc(n)$ are such that $\dH(\pi,\rho)=n-1$, then it
follows that $\pi\wedge\rho=0_{n}$ and $\pi\vee\rho=1_{n}$, and
also that $|\pi|+|\rho|=n+1$. 
\end{enumerate}
\end{prop}

\begin{proof}
In $(2)$, the inequality ``$\geq$'' is an immediate consequence
of the fact that every edge in $\cH_{n}$ connects partitions on adjacent
levels of the graph. In order to prove ``$\leq$'', we take a saturated
increasing chain in $\nc(n)$ which goes from $\pi$ to $\rho$ and
we observe that it gives a path of length $|\pi|-|\rho|$ which connects
the two partitions.

For $\left(1\right)$, note that $d_{H}\left(0_{n},1_{n}\right)=n-1$
so the diameter is at least $n-1$. On the other hand, for every $\pi,\rho\in\nc(n)$
we have $\pi\wedge\rho\leq\pi$ and $\pi\leq\pi\vee\rho$, so part
$(2)$ assures us that 
\[
\dH(\pi\wedge\rho,\pi)=|\pi\wedge\rho|-|\pi|\mbox{ and }\dH(\pi,\pi\vee\rho)=|\pi|-|\pi\vee\rho|.
\]
Similar inequalities hold with $\rho$ featured in the place of $\pi$.
But then we can write 
\begin{eqnarray}
\dH\left(\pi,\rho\right) & \le & \frac{\dH\left(\pi,\pi\wedge\rho\right)+\dH\left(\pi\wedge\rho,\rho\right)+\dH\left(\pi,\pi\lor\rho\right)+\dH\left(\pi\lor\rho,\rho\right)}{2}\nonumber \\
 & = & \left|\pi\wedge\rho\right|-\left|\pi\vee\rho\right|\le n-1,\label{eq:d_H le n-1}
\end{eqnarray}
where the latter inequality simply holds because $|\pi\wedge\rho|\leq n$
and $|\pi\vee\rho|\geq1$.

To show $\left(3\right)$, note that from \eqref{eq:d_H le n-1} it
is clear that the equality $\dH(\pi,\rho)=n-1$ forces the equalities
$|\pi\wedge\rho|=n$, $|\pi\vee\rho|=1$, i.e.~$\pi\wedge\rho=0_{n}$
and $\pi\vee\rho=1_{n}$. Finally, if $\pi,\rho$ are such that $\dH(\pi,\rho)=n-1$,
then by writing 
\[
n-1=\dH(\pi,\rho)\leq\dH(\pi,1_{n})+\dH(1_{n},\rho)=|\pi|+|\rho|-2
\]
we obtain that $|\pi|+|\rho|\geq n+1$, and by writing 
\[
n-1=\dH(\pi,\rho)\leq\dH(\pi,0_{n})+\dH(0_{n},\rho)=2n-\bigl(\,|\pi|+|\rho|\,\bigr)
\]
we obtain that $|\pi|+|\rho|\leq n+1$. Hence the equality $\dH(\pi,\rho)=n-1$
implies that $|\pi|+|\rho|=n+1$. 
\end{proof}
\begin{rem}
\label{rem:28}The combination of necessary conditions $\pi\wedge\rho=0_{n},\pi\vee\rho=1_{n}$
and $|\pi|+|\rho|=n+1$ found in Proposition \ref{prop:diameter of H_n}(3)
is not sufficient to ensure that $\dH(\pi,\rho)=n-1$. For example,
the following piece (subgraph) from $\cH_{6}$ shows two non-crossing
partitions at distance $3$ which satisfy these three properties:
\[
\xymatrix{1346,2,5\ar@{-}[d] & 16,25,34\ar@{-}[d]\\
13,2,46,5 & 16,2,34,5\ar@{-}[ul]
}
\]
\end{rem}

\subsubsection*{Embedding $\nc\left(n\right)$ in $\cS_{n}$ and the Kreweras complementation
map}
\begin{defn}
\label{def:embedding in S_n}Let $n$ be a positive integer and let
$\cS_{n}$ be the group of permutations of $\left\{ 1,\ldots,n\right\} $.
For every $\pi\in\nc(n)$ we construct an associated permutation \marginpar{$P_{\pi}$}$P_{\pi}\in\cS_{n}$
as follows: the blocks of $\pi$ become orbits of $P_{\pi}$, and
$P_{\pi}$ performs an increasing cycle on every such block; that
is, if $V=\{i_{1},i_{2},\ldots,i_{k}\}\in\pi$ with $i_{1}<i_{2}<\cdots<i_{k}$,
then we have $P_{\pi}(i_{1})=i_{2},\ldots,P_{\pi}(i_{k-1})=i_{k},\,P_{\pi}(i_{k})=i_{1}$.
This description includes the fact that if $\left\{ i\right\} =V\in\pi$
is a singleton, then $P_{\pi}\left(i\right)=i$. 
\end{defn}

The map $\nc(n)\ni\pi\mapsto P_{\pi}\in\cS_{n}$ is obviously injective.
It is also clear that $\#(P_{\pi})=\left|\pi\right|$ for every $\pi\in\nc\left(n\right)$,
where for $\sigma\in\cS_{n}$ we denote by $\#\left(\sigma\right)$
the number of cycles in $\sigma$. This embedding was introduced in
\cite{biane1997some}, and it has additional nice properties, some
of which are mentioned in Section \ref{sec:formulas-for-d_H}. 
\begin{rem}
\label{rem:kreweras}The Hasse diagram $\cH_{n}$ displays top-down
symmetry. This is explained by the existence of a natural bijection
$\Kr_{n}\colon\nc\left(n\right)\to\nc\left(n\right)$\marginpar{$\Kr$}
which reverses partial refinement order: for $\pi,\rho\in\nc\left(n\right)$
one has 
\[
(\pi\leq\rho)\ \Leftrightarrow\ \left(\Kr_{n}(\rho)\leq\Kr_{n}(\pi)\right).
\]
The bijection $\Kr_{n}$ was introduced by Kreweras \cite{kreweras1972partitions}
and is called the \emph{Kreweras complementation map}. It can be defined
in terms of permutations: 
\begin{equation}
P_{\Kr_{n}(\pi)}=P_{\pi}^{-1}P_{1_{n}},\ \ \forall\,\pi\in NC(n).\label{eqn:Kreweras via S_n}
\end{equation}
Note that in \eqref{eqn:Kreweras via S_n}, $P_{1_{n}}$ is the long
cycle $\left(1~2~3~\ldots~n\right)$. For a discussion about this
and a proof that for $\pi\in\nc\left(n\right)$, $P_{\pi}^{-1}P_{1_{n}}$
is equal indeed to $P_{\lambda}$ for some $\lambda\in\nc\left(n\right)$,
consult e.g.~\cite[Exercises 18.25 and 18.26]{nica-speicher2006lectures}.
For an alternative definition of the map $\Kr$ see e.g.~\cite[Pages 147-148]{nica-speicher2006lectures}
and also Remark \ref{rem:|NC(n)|=00003Dcat  and def of kr} below.
The fact that $\Kr_{n}$ is indeed a top-down symmetry follows from
\eqref{eqn:Kreweras via S_n}, together with Remark \ref{rem:|NC(n)|=00003Dcat  and def of kr}
and Theorem \ref{thm:equivalences for d_H} below. 
\end{rem}

\section{{\large{}Meandric Systems and Their Relation to \boldmath{$\nc\left(n\right)$}
\label{sec:Meandric-systems}}}

Let $L$ be a fixed oriented line in the Euclidean plane $\mathbb{E}:=\mathbb{R}^{2}$
with $2n$ marked points $p_{1},\ldots,p_{2n}$. Consider a non-intersecting
(not necessarily connected) closed curve $C\subseteq\mathbb{E}$,
which transversely intersects the line $L$ at precisely the points
$p_{1},\ldots,p_{2n}$. Say that two such curves $C_{1}$ and $C_{2}$
are equivalent if they can be deformed into each other by an isotopy
of $\mathbb{E}$ which fixes the line $L$ pointwise. The equivalence
class $M=\left[C\right]$ is called a \emph{meandric system} of order
$n$, or simply \emph{a meander} of order $n$ if $C$ is connected.
We denote the set of all meandric systems of order $n$ by $\fM\left(n\right)$\marginpar{$\fM\left(n\right)$}.

The notoriously difficult problem of enumerating meanders was first
introduced by Lando and Zvonkin \cite{Lando1992meanders}. It emerges
in a variety of different areas inside and outside Mathematics, see
for example \cite{arnol1988branched,di1996meanders,lando2004graphs}.
Figure \ref{fig:meanders of order 3} shows all meanders of order
$3$ and Figure \ref{fig:meandric system of order 4} illustrates
a meandric system of order $4$ with two connected components.

\begin{figure}
\centering{}\includegraphics[viewport=150bp 400bp 450bp 600bp,scale=0.7]{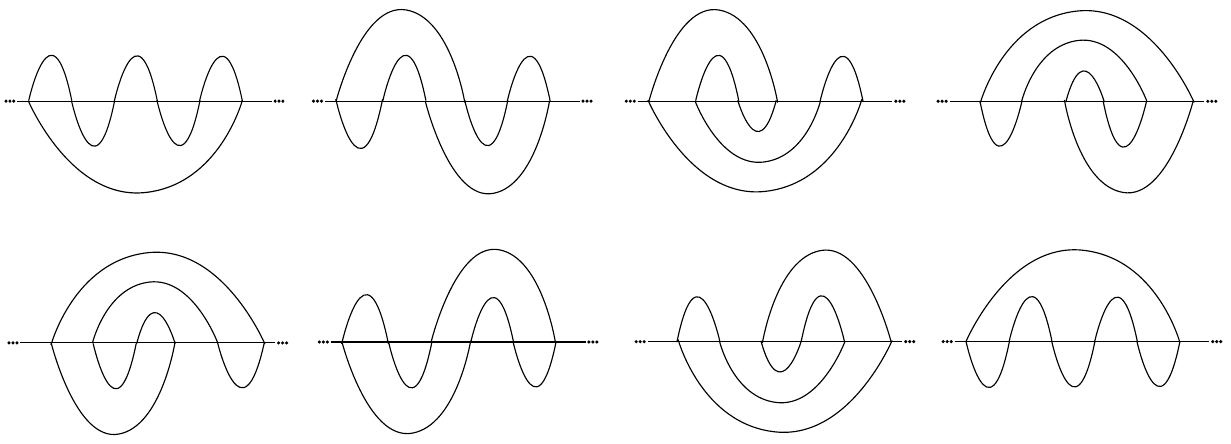}\caption{\label{fig:meanders of order 3}All meanders of order $3$}
\end{figure}

\begin{figure}[h]
\centering{}\includegraphics[viewport=200bp 510bp 300bp 620bp]{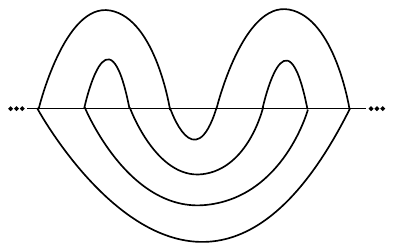}\caption{\label{fig:meandric system of order 4}A meandric system of order
$4$ with two connected components}
\end{figure}

Consider the two open half-planes defined by $L$. The intersection
of a meander with each half plane is a collection of $n$ disjoint,
self-avoiding arcs, each of which connects two of the points $p_{1},\ldots,p_{2n}$.
Such a configuration is called a \emph{non-crossing pairing (or arch-diagram)}
\emph{of order $n$}. The set of non-crossing pairings of order $n$,
considered up to homeomorphism as above, is denoted \marginpar{\emph{$\ncp\left(n\right)$}}
\emph{$\ncp\left(n\right)$}, and\emph{ }it is standard that the cardinality
of $\mathrm{\nc P}\left(n\right)$ is $\Cat_{n}$, the $n$-th Catalan\footnote{Of course, $\ncp\left(n\right)$ is also a subset of $\nc\left(2n\right)$
consisting of non-crossing partitions with all blocks having size
$2$, but we do not use this point of view here.} number. Evidently, there is a bijection between pairs of non-crossing
pairings of order $n$ and meandric systems of order $n$, given by
\begin{equation}
\ncp\left(n\right)^{2}\ni\left(\pi,\rho\right)\mapsto M\left(\pi,\rho\right)\in\fM\left(n\right),\label{eq:(pi,rho) to M(pi,rho)}
\end{equation}
where, by convention, the line $L$ is assumed to be horizontal, $\pi$
corresponds to the pairing in the upper half-plane, and $\rho$ in
the lower half-plane. The points $p_{1},\ldots,p_{2n}$ appear from
left to right. For example, in the meandric system in Figure \ref{fig:meandric system of order 4},
$\pi$ is the pairing $\left\{ \left\{ p_{1},p_{4}\right\} ,\left\{ p_{2},p_{3}\right\} ,\left\{ p_{5},p_{8}\right\} ,\left\{ p_{6},p_{7}\right\} \right\} \in\ncp\left(4\right)$,
while $\rho$ is $\left\{ \left\{ p_{1},p_{8}\right\} ,\left\{ p_{2},p_{7}\right\} ,\left\{ p_{3},p_{6}\right\} ,\left\{ p_{4},p_{5}\right\} \right\} \in\ncp\left(4\right)$.

The number of curves, or components, of the meandric system $M\left(\pi,\rho\right)$
is denoted \marginpar{$\#M\left(\pi,\rho\right)$}$\#M\left(\pi,\sigma\right)$,
and satisfies $1\le\#M\left(\pi,\sigma\right)\le n$. Of course, $\#M\left(\pi,\sigma\right)=1$
if and only if $M\left(\pi,\rho\right)$ is a meander. The total number
of meandric systems of order $n$ is, therefore, $\Cat_{n}^{~2}$.
However, very little is known about the distribution of $\#M\left(\pi,\sigma\right)$
when $n$ is large. Some numerics can be found in \cite{di1997meander}.

\subsubsection*{A bijection between non-crossing partitions and non-crossing pairings}

There is a well-known natural bijection between $\nc\left(n\right)$,
the set of non-crossing partitions of order $n$, and $\ncp\left(n\right)$,
the set of non-crossing pairings of order $n$. The most intuitive
way to describe this bijection is geometrical: given a non-crossing
partition $\pi\in\nc\left(n\right)$ drawn as in Figure \ref{fig:example in NC(9)},
take an $\varepsilon$-neighborhood $N_{\varepsilon}\left(\pi\right)$
of the graph describing $\pi$ for small enough $\varepsilon>0$,
and consider $\partial\left(\pi\right)$, defined to be the intersection
of the boundary of $N_{\varepsilon}\left(\pi\right)$ with the upper
half-plane. Observe that $\partial\pi$ is a non-crossing pairing
of order $n$, with $2n$ points on the horizontal line, one on each
side of each of the original $n$ points. The passage from $\pi\in NC(n)$
to $\partial\left(\pi\right)\in\ncp\left(n\right)$ is commonly known
as the ``doubling'' or ``fattening'' of the partition $\pi$.
This is illustrated in Figure \ref{fig:from nc to ncp}.

\begin{figure}
\centering{}\includegraphics[viewport=120bp 310bp 150bp 370bp,scale=2]{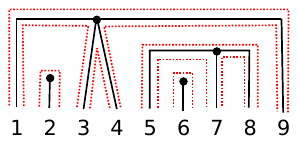}\caption{\label{fig:from nc to ncp}The non-crossing pairing (dashed lines)
in $\ncp\left(9\right)$ corresponding to the non-crossing partition
$\left\{ \left\{ 1,3,4,9\right\} ,\left\{ 2\right\} ,\left\{ 5,7,8\right\} ,\left\{ 6\right\} \right\} \in\nc\left(9\right)$}
\end{figure}

Conversely, given a non-crossing pairing $\rho\in\ncp\left(n\right)$
consisting of $n$ disjoint curves with endpoints $p_{1},\ldots,p_{2n}$
as above, mark additional $n$ points $q_{1},\ldots,q_{n}$ so that
$q_{i}$ lies on the interval of $L$ between $p_{2i-1}$ and $p_{2i}$.
The $n$ curves of $\rho$ cut the upper-half plane into $n+1$ connected
components. The connected components containing $q$-points define
a non-crossing partition $c\left(\rho\right)\in\nc\left(n\right)$,
where every block consists (of the indices) of the $q$-points lying
inside one of these components. This is illustrated in Figure \ref{fig:from ncp to nc}.
The fact that $c\left(\rho\right)$ is non-crossing is immediate.
The fact that $\rho\mapsto c\left(\rho\right)$ is the inverse map
to $\pi\mapsto\partial\left(\pi\right)$ follows easily from the observation
that if $p_{i}$ is connected to $p_{j}$ then one of $i$ and $j$
is odd and the other is even. 
\begin{rem}
\label{rem:|NC(n)|=00003Dcat  and def of kr} 
\begin{enumerate}
\item The bijection between $\nc\left(n\right)$ and $\ncp\left(n\right)$
yields a proof of the fact mentioned in Definition \ref{def:nc} that
$\left|\nc\left(n\right)\right|=\Cat_{n}$. 
\item In the description of $c\left(\rho\right)$, if the point $q_{i}$
is taken to lie in the interval $\left(p_{2i},p_{2i+1}\right)$, instead
of in the interval $\left(p_{2i-1},p_{2i}\right)$ (with $q_{n}$
lying to the right of $p_{2n}$), then the resulting non-crossing
partition $c'\left(\rho\right)$ is precisely the image of $c\left(\rho\right)$
under the Kreweras complementation map: $c'\left(\rho\right)=\Kr\left(c\left(\rho\right)\right)$.
Because the $n$ curves of $\rho$ cut the upper half-space into $n+1$
connected components, it follows that $\left|\pi\right|+\left|\Kr\left(\pi\right)\right|=n+1$
for every $\pi\in\nc\left(n\right)$. It is also easy to infer from
this graphical definition of $\Kr$ why \eqref{eqn:Kreweras via S_n}
holds. 
\end{enumerate}
\end{rem}

\begin{figure}
\centering{}\includegraphics[viewport=150bp 400bp 500bp 580bp,scale=0.7]{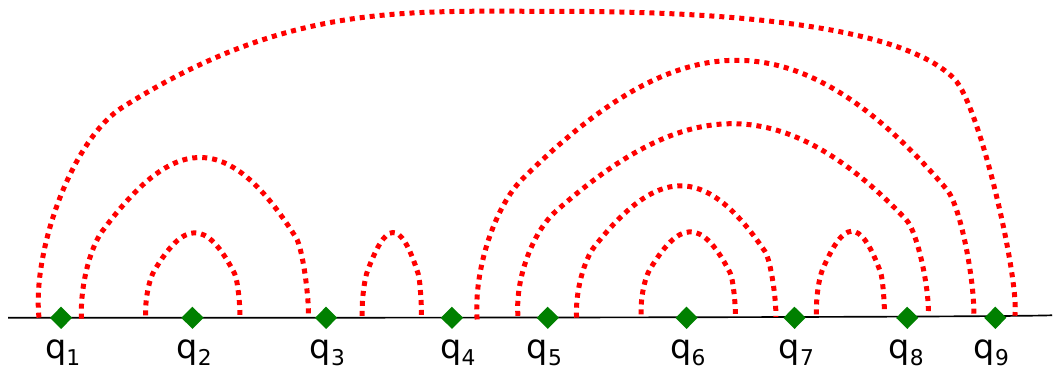}\caption{\label{fig:from ncp to nc}How a non-crossing pairing (dashed lines)
in $\ncp\left(9\right)$ determines a non-crossing partition of the
same order. In this example the resulting non-crossing partition is
$\left\{ \left\{ 1,3,4,9\right\} ,\left\{ 2\right\} ,\left\{ 5,7,8\right\} ,\left\{ 6\right\} \right\} \in\nc\left(9\right)$.}
\end{figure}

\begin{defn}
\label{def:M(pi,rho)}For $\pi,\rho\in\nc\left(n\right)$, we denote
by \marginpar{$M\left(\pi,\rho\right)$} 
\[
M\left(\pi,\rho\right):=M\left(\partial\left(\pi\right),\partial\left(\rho\right)\right)
\]
the meandric system corresponding to $\partial\left(\pi\right),\partial\left(\rho\right)\in\ncp\left(n\right)$.
As before, $\#M\left(\pi,\rho\right)$ denotes the number of components
of the meandric system. 
\end{defn}

Note the abuse of notation here: we use the same notation for the
bijection $\ncp\left(n\right)^{2}\to\fM\left(n\right)$ and for the
bijection $\nc\left(n\right)^{2}\to\fM\left(n\right)$.

\section{{\large{}Equivalent Expressions for \boldmath{$\dH(\pi,\rho)$}
\label{sec:formulas-for-d_H}}}

In this section we review some equivalent definitions, expressions
and formulas for the quantity $d_{H}\left(\pi,\rho\right)$, as well
as different criteria for $M\left(\pi,\rho\right)$ being a meander.
All these equivalences can be found in the literature, although to
the best of our knowledge, not in a single reference. We state the
equivalences, refer to proofs in the literature and describe a few
of the arguments.

The following two definitions are central to some of the equivalences,
and the first one plays a crucial role in the subsequent sections: 
\begin{defn}
\label{def:Gamma(pi,rho)}Given $\pi,\rho\in\nc\left(n\right)$, let
$\Gamma\left(\pi,\rho\right)$\marginpar{$\Gamma\left(\pi,\rho\right)$}
denote the planar, edge-labeled and vertex-colored graph obtained
by drawing $\pi$ in the upper half-plane as in Figure \ref{fig:example in NC(9)},
and $\rho$ in the lower half plane, with the same $n$ points labeled
$1,\ldots,n$. More precisely, mark $n$ points along an invisible
horizontal line and label them, from left to right, by $1,\ldots,n$,
draw $\pi$ in the upper half-plane with a vertex for every block
$V\in\pi$, and $\rho$ in the lower half-plane with a vertex for
every block $W\in\rho$. These are the $\left|\pi\right|+\left|\rho\right|$
vertices of the graph. Its $n$ edges consist of the $n$ lines incident
with the marked points, and are labeled accordingly. That is, for
every $i=1,\ldots,n$, there is an edge labeled $i$ connecting the
block in $\pi$ containing $i$ with the block in $\rho$ containing
$i$. By convention, the vertices of $\pi$ are black and the vertices
of $\rho$ are white\footnote{This convention will be useful only in subsequent sections.}.
This is illustrated in Figure \ref{fig:Gamma(pi,rho)}. To the best
of our knowledge, this graph was first introduced in \cite{franz1998partial}. 
\end{defn}

\begin{figure}
\centering{}\includegraphics[viewport=100bp 250bp 300bp 400bp]{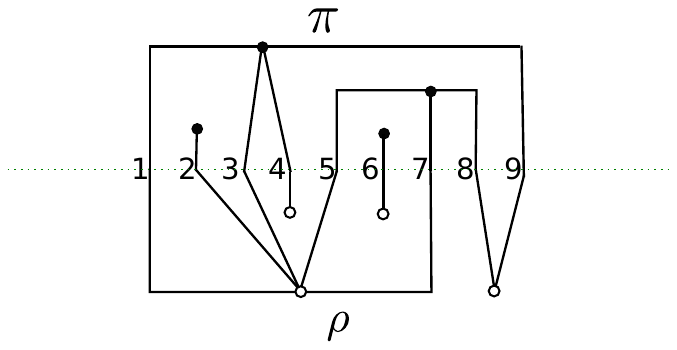}\caption{\label{fig:Gamma(pi,rho)}The planar, edge-labeled and vertex-colored
graph $\Gamma\left(\pi,\rho\right)$ (in solid lines) for two partitions
in $\nc\left(9\right)$: $\pi=\left\{ \left\{ 1,3,4,9\right\} ,\left\{ 2\right\} ,\left\{ 5,7,8\right\} ,\left\{ 6\right\} \right\} $
and $\rho=\left\{ \left\{ 1,2,3,5,7\right\} ,\left\{ 4\right\} ,\left\{ 6\right\} ,\left\{ 8,9\right\} \right\} $}
\end{figure}

\begin{rem}
\label{rem:veetild}For $\pi,\rho\in\nc\left(n\right)$, we let \marginpar{$\pi\veetild\rho$}$\pi\veetild\rho$
denote the smallest common upper bound of $\pi$ and $\rho$ in the
lattice of all partitions of $\left\{ 1,\ldots,n\right\} $, crossing
and non-crossing alike. It is easy to see that two numbers $i,j\in\{1,\ldots,n\}$
belong to the same block of $\pi\veetild\rho$ if and only if there
exist $k\in\bN$ and $i_{0},i_{1},\ldots,i_{2k}\in\{1,\ldots,n\}$
such that $i=i_{0}\ecpi i_{1}\ecrho i_{2}\ecpi\cdots\ecpi i_{2k-1}\ecrho i_{2k}=j$,
where $i\stackrel{\pi}{\sim j}$ means $i$ and $j$ belong to the
same block of $\pi$. Evidently, the blocks of $\pi\veetild\rho$
are in one-to-one correspondence with the connected components of
$\Gamma\left(\pi,\rho\right)$: the elements of a block correspond
to the edge-labels in a component. 
\end{rem}

\begin{defn}
\label{def:norm of perm}For a permutation $\sigma\in\cS_{n}$, denote
by \marginpar{$\left\Vert \sigma\right\Vert $}$\left\Vert \sigma\right\Vert $
its norm, defined as the length of the shortest product of transpositions
in $\cS_{n}$ giving $\sigma$. Namely, 
\[
\left\Vert \sigma\right\Vert \stackrel{\mathrm{def}}{=}\min\left\{ g\,\middle|\,\sigma=t_{1}\cdots t_{g},~\mathrm{where~}t_{1},\ldots,t_{g}~\mathrm{are~transpositions}\right\} .
\]
It is standard that $\left\Vert \sigma\right\Vert =n-\#\mathrm{cycles}\left(\sigma\right)$
(e.g.~\cite[Proposition 2.4]{hall2006meanders}). 
\end{defn}

We can now state the equivalent expressions for the distance $d_{H}\left(\pi,\rho\right)$
between two non-crossing partitions in the Hasse diagram $\cH_{n}$.
The equivalence of \ref{enu:d_H}, \ref{enu:n-M} and \ref{enu:perms}
from Theorem \ref{thm:equivalences for d_H} appears in \cite[Theorem 3.3]{hall2006meanders}
and independently in \cite{savitt2009polynomials}. The equivalence
of \ref{enu:n-M}, of \ref{enu:|pi|+|rho|-2pi_0} and, in a slightly
different language of \ref{enu:sum over con-com}, is shown in \cite{franz1998partial}. 
\begin{thm}
\label{thm:equivalences for d_H}Let $\pi,\rho\in\nc\left(n\right)$
be two non-crossing partitions of order $n$. Then the following quantities
are equal: 
\begin{enumerate}
\item \label{enu:d_H}$d_{H}\left(\pi,\rho\right)$ 
\item \label{enu:n-M}$n-\#M\left(\pi,\rho\right)$ 
\item \label{enu:perms}$\left\Vert P_{\pi}P_{\rho}^{-1}\right\Vert =n-\#\mathrm{cycles}\left(P_{\pi}P_{\rho}^{-1}\right)$ 
\item \label{enu:|pi|+|rho|-2pi_0}$\left|\pi\right|+\left|\rho\right|-2\left|\pi\veetild\rho\right|$ 
\item \label{enu:sum over con-com}${\displaystyle n-\sum_{\begin{gathered}C\colon{\scriptstyle \mathrm{connected}}\\
{\scriptstyle \mathrm{component~of~}\Gamma\left(\pi,\rho\right)}
\end{gathered}
}\#\mathrm{faces}\left(C\right)}$ 
\end{enumerate}
\end{thm}

\begin{proof}
For a short and elegant proof of the equivalence \ref{enu:d_H}$\Longleftrightarrow$\ref{enu:perms}
we refer the reader to \cite[Theorem 3.3]{hall2006meanders}. It is
an easy observation that the different components of the meandric
system $M\left(\pi,\rho\right)$ are in one-to-one correspondence
with the cycles of $P_{\pi}P_{\rho}^{-1}$: starting at a point on
the horizontal invisible line just left to some edge $i$ in $\Gamma\left(\pi,\rho\right)$,
the permutation $P_{\rho}^{-1}$ takes $i$ along the arc in the lower
half-plane, and then $P_{\pi}$ takes the resulting point through
an arc in the upper half-plane. Consult, for example, Figure \ref{fig:Gamma(pi,rho) with meandric system}.
This shows the equivalence \ref{enu:n-M}$\Longleftrightarrow$\ref{enu:perms}.

It is clear that the components of $M\left(\pi,\rho\right)$ are curves
tracing faces (in the sense of planar graphs) of $\Gamma\left(\pi,\rho\right)$,
which shows \ref{enu:n-M}$\Longleftrightarrow$\ref{enu:sum over con-com}.
Recall Euler's formula for a connected planar graph $C$: 
\[
v\left(C\right)-e\left(C\right)+f\left(C\right)=2,
\]
where $v\left(C\right)$ is the number of vertices of a graph, $e\left(C\right)$
the number of edges and $f\left(C\right)$ the number of faces. Thus,
recalling the fact from Remark \ref{rem:veetild} that $\left|\pi\veetild\rho\right|$
is the number of components of $\Gamma\left(\pi,\rho\right)$, we
obtain: 
\begin{eqnarray*}
n-\!\!\!\!\!\!\!\!\!\!\!\!\sum_{\begin{gathered}{\scriptstyle C:\mathrm{connected}}\\
{\scriptstyle \mathrm{component~of~}\Gamma\left(\pi,\rho\right)}
\end{gathered}
}\!\!\!\!\!\!\!\!\!\!\!\!f\left(C\right) & = & e\left(\Gamma\left(\pi,\rho\right)\right)-\!\!\!\!\!\!\!\!\!\!\!\!\sum_{\begin{gathered}{\scriptstyle C:\mathrm{connected}}\\
{\scriptstyle \mathrm{component~of~}\Gamma\left(\pi,\rho\right)}
\end{gathered}
}\!\!\!\!\!\!\!\!\!\!\!\!\left[2+e\left(C\right)-v\left(C\right)\right]\\
 & = & v\left(\Gamma\left(\pi,\rho\right)\right)-\!\!\!\!\!\!\!\!\!\!\!\!\sum_{\begin{gathered}{\scriptstyle C:\mathrm{connected}}\\
{\scriptstyle \mathrm{component~of~}\Gamma\left(\pi,\rho\right)}
\end{gathered}
}\!\!\!\!\!\!\!\!\!\!\!\!2\\
 & = & \left|\pi\right|+\left|\rho\right|-2\left|\pi\veetild\rho\right|,
\end{eqnarray*}
whence \ref{enu:|pi|+|rho|-2pi_0}$\Longleftrightarrow$\ref{enu:sum over con-com}. 
\end{proof}
\begin{figure}
\centering{}\includegraphics[viewport=100bp 250bp 300bp 400bp]{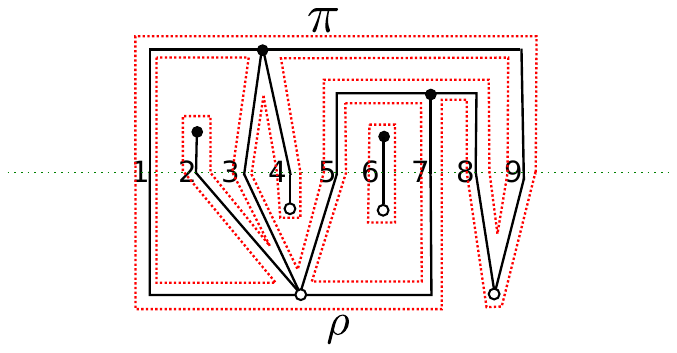}\caption{\label{fig:Gamma(pi,rho) with meandric system}The meandric system
$M\left(\pi,\rho\right)$ (in dashed lines) as obtained from $\Gamma\left(\pi,\rho\right)$,
the graph from Figure \ref{fig:Gamma(pi,rho)} (in solid lines); $\#M\left(\pi,\rho\right)=5$}
\end{figure}

\begin{rem}
\begin{enumerate}
\item An equivalent way of defining the norm of a permutation is through
a Cayley graph of $\cS_{n}$, as follows. Let $\mathrm{T}_{n}$ be
the set of $\binom{n}{2}$ transpositions in $\cS_{n}$. The Cayley
graph $\mathrm{Cay}\left(\cS_{n},\mathrm{T}_{n}\right)$ is a graph
with set of vertices corresponding to the elements of $\cS_{n}$,
and an edge between the permutations $\sigma$ and $\tau$ whenever
$\sigma\tau^{-1}\in\mathrm{T}_{n}$. Denote by $d_{C}\left(\sigma,\tau\right)$
the distance between two permutations in this graph. It is obvious
that $\left\Vert \sigma\right\Vert =d_{C}\left(e,\sigma\right)$ where
$e\in\cS_{n}$ is the identity, and, more generally, that $d_{C}\left(\sigma,\tau\right)=\left\Vert \sigma\tau^{-1}\right\Vert $.
This means that the equivalence \ref{enu:d_H}$\Longleftrightarrow$\ref{enu:perms}
from Theorem \ref{thm:equivalences for d_H} can be interpreted as
saying that the embedding $\nc\left(n\right)\hookrightarrow\cS_{n}$
from Definition \ref{def:embedding in S_n} is an isometry: $d_{H}\left(\pi,\rho\right)=d_{C}\left(P_{\pi},P_{\rho}\right)$. 
\item A sixth equivalent quantity can be added to the list in Theorem \ref{thm:equivalences for d_H}:
a natural distance defined on the set $\ncp\left(n\right)$ of non-crossing
pairings by saying that $\pi,\rho\in\ncp\left(n\right)$ are neighbors
if and only if one can be obtained from the other by a simple switch
$\left\{ a,b\right\} ,\left\{ c,d\right\} \mapsto\left\{ a,d\right\} ,\left\{ b,c\right\} $.
See \cite{savitt2009polynomials} for more details. 
\end{enumerate}
\end{rem}

As an immediate corollary from Theorem \ref{thm:equivalences for d_H}
we get the following equivalent criteria for $M\left(\pi,\rho\right)$
being a meander: 
\begin{cor}
\label{cor:equivalence for meander}Let $\pi,\rho\in\nc\left(n\right)$
be two non-crossing partitions of order $n$. Then the following are
equivalent: 
\begin{enumerate}
\item \label{enu:meander}$M\left(\pi,\rho\right)$ is a meander 
\item \label{enu:d_H=00003Dn-1}$d_{H}\left(\pi,\rho\right)=n-1$ namely,
$\pi$ and $\rho$ define a diameter in $\cH_{n}$ 
\item \label{enu:norm=00003Dn-1}$\left\Vert P_{\pi}P_{\rho}^{-1}\right\Vert =n-1$,
namely, $P_{\pi}P_{\rho}^{-1}$ is an $n$-cycle 
\item \label{enu:two conds}$\left|\pi\right|+\left|\rho\right|=n+1$ and
$\left|\pi\veetild\rho\right|=1$ 
\item \label{enu:tree}$\Gamma\left(\pi,\rho\right)$ is a tree 
\end{enumerate}
\end{cor}

\begin{proof}
The only equivalence which is not completely obvious from Theorem
\ref{thm:equivalences for d_H} is that of \ref{enu:two conds} with
the others, so we prove here \ref{enu:two conds}$\Longleftrightarrow$\ref{enu:tree}.
Indeed, since the graph $\Gamma\left(\pi,\rho\right)$ has exactly
$n$ edges, it is a tree if and only if it is connected ($\left|\pi\veetild\rho\right|=1$)
and has exactly $n+1$ vertices ($\left|\pi\right|+\left|\rho\right|=n+1$). 
\end{proof}
The following corollary shows that every non-crossing partition has
a meandric partner\footnote{The mere fact that every $\pi\in\nc\left(n\right)$ has a meandric
partner is not hard: for example, one can easily find $\rho\in\nc\left(n\right)$
whose diagram in the lower-half plane completes the diagram of $\pi$
in the upper half-plane into a tree.} which is, in some sense, canonical. The statement of this corollary
may be known, but we are not aware of a reference. 
\begin{cor}
\label{cor:pi,kr(pi) is meander}For every $\pi\in\nc\left(n\right)$,
$M\left(\pi,\Kr\left(\pi\right)\right)$ is a meander, or, equivalently\linebreak{}
 $d_{H}\left(\pi,\Kr\left(\pi\right)\right)=n-1$. 
\end{cor}

\begin{proof}
We show that Criterion \ref{enu:two conds} from Corollary \ref{cor:equivalence for meander}
holds. Indeed, the well-known fact that $\left|\pi\right|+\left|\Kr\left(\pi\right)\right|=n+1$
was explained in Remark \ref{rem:|NC(n)|=00003Dcat  and def of kr}.
To see that $\left|\pi\veetild\Kr\left(\pi\right)\right|=1$, note
that for every $i=2,\ldots,n$, 
\[
i\stackrel{\pi}{\sim}P_{\pi}^{-1}\left(i\right)\stackrel{\Kr\left(\pi\right)}{\sim}i-1
\]
because, by \eqref{eqn:Kreweras via S_n}, $P_{\Kr\left(\pi\right)}\left(i-1\right)=\left[P_{\pi}^{-1}\left(1~2~\ldots~n\right)\right]\left(i-1\right)=P_{\pi}^{-1}\left(i\right)$.
Hence $i$ and $i-1$ belong to the same block of $\pi\veetild\Kr\left(\pi\right)$. 
\end{proof}
\begin{cor}
\label{cor:lower-bound for average dist}For every $n\in\bN$ one
has that 
\begin{equation}
\frac{1}{\Cat_{n}^{2}}\ \sum_{\pi,\rho\in\nc(n)}\,\dH(\pi,\rho)\geq\frac{n-1}{2}.
\end{equation}
\end{cor}

\begin{proof}
We fix a partition $\rho\in\nc\left(n\right)$ and we use the triangle
inequality and Corollary \ref{cor:pi,kr(pi) is meander} to infer
that: 
\[
\dH(\pi,\rho)+\dH\left(\Kr_{n}(\pi),\rho\right)\geq\dH\left(\pi,\Kr_{n}(\pi)\right)=n-1,\ \ \forall\,\pi\in\nc\left(n\right).
\]
Summing over $\pi$ gives 
\[
\left(\sum_{\pi\in\nc(n)}\dH\left(\pi,\rho\right)\right)+\left(\sum_{\pi\in\nc(n)}\dH\left(\Kr_{n}(\pi),\rho\right)\right)\geq\left(n-1\right)\cdot\Cat_{n}.
\]
However, the two sums on the left-hand side of the preceding equation
are equal to each other (because $\Kr_{n}$ is bijective), so what
we have obtained is that, for our fixed $\rho$: 
\[
\sum_{\pi\in\nc(n)}\dH(\pi,\rho)\geq\frac{n-1}{2}\cdot\Cat_{n}.
\]
Finally, we let $\rho$ vary in $\nc(n)$ and sum over it, and the
required formula follows. 
\end{proof}

\section{Interval Partitions and Meanders with Shallow Top\label{sec:Interval-Partitions}}
\begin{defn}
\label{def:interval}A partition $\pi$ of $\{1,\ldots,n\}$ is said
to be an \emph{interval partition} when every block of $\pi$ is of
the form $\left\{ i\in\mathbb{N}\,\middle|\,p\leq i\leq q\right\} $
for some $p\leq q$ in $\left\{ 1,\ldots,n\right\} $. We denote the
set of all interval partitions of $\left\{ 1,\ldots,n\right\} $ by
$\Int\left(n\right)$. It is obvious that $\Int(n)\subseteq\nc\left(n\right)$,
and it is easy to count that $\left|\Int(n)\right|=2^{n-1}$. Some
of the arguments and results below work, in fact, for the larger collection
\marginpar{$\intc\left(n\right)$}$\intc\left(n\right)$ which includes
all cyclic permutations of interval partitions. This is still a subset
of $\nc\left(n\right)$, with $\left|\intc\left(n\right)\right|=2^{n}-n$.

If $\pi\in\Int\left(n\right)$, we say that the associated non-crossing
pairing $\partial\pi\in\ncp\left(n\right)$ is \emph{shallow. }We
say that the meandric system $M\left(\pi,\rho\right)$ with $\pi,\rho\in\nc\left(n\right)$
has a\marginpar{shallow top} \emph{shallow top} if $\pi\in\Int\left(n\right)$
is an interval partition. 
\end{defn}

Note that $\partial\pi\in\ncp\left(n\right)$ being shallow is equivalent
to the fact that its diagram, as in Figure \ref{fig:from ncp to nc},
can be drawn with at most two curves above every point on the infinite
horizontal line. This is the rationale behind the term ``shallow''. 
\begin{prop}
\label{prop:d_H for interval partitions}Let $n\in\mathbb{N}$ and
consider partitions $\pi\in\intc\left(n\right)$ and $\rho\in\nc\left(n\right)$.
Then $\pi\vee\rho=\pi\veetild\rho$, and consequently 
\begin{equation}
\dH(\pi,\rho)=|\pi|+|\rho|-2|\pi\vee\rho|.\label{eqn:d_H for interval partition}
\end{equation}
\end{prop}

\begin{proof}
The equality \eqref{eqn:d_H for interval partition} follows from
$\pi\vee\rho=\pi\veetild\rho$ together with the equivalence \ref{enu:d_H}$\Longleftrightarrow$\ref{enu:|pi|+|rho|-2pi_0}
in Theorem \ref{thm:equivalences for d_H}, so it is enough to prove
that $\pi\vee\rho=\pi\veetild\rho$. Assume first that $\pi$ consists
of $n-1$ blocks: a block of size two and $n-2$ singletons. By symmetry,
we can assume the block of size two is $\left\{ 1,n\right\} $. We
claim that for every $\rho\in\nc\left(n\right)$, the partition $\pi\veetild\rho$
is non-crossing. Indeed, $\pi\veetild\rho$ is the partition obtained
from $\rho$ by merging the block containing $1$ with the one containing
$n$. If $1\stackrel{\rho}{\sim}n$, then $\pi\veetild\rho=\rho$
and the claim is clear. Otherwise, assume the block of $1$ in $\rho$
is $B=\left\{ a_{1}=1,a_{2},\ldots,a_{k}\right\} $ and the block
of $n$ is $B'=\left\{ c_{1},\ldots,c_{\ell-1},c_{\ell}=n\right\} $
with 
\[
1<a_{2}<a_{3}<\ldots<a_{k}<c_{1}<c_{2}<\ldots<c_{\ell-1}<n.
\]
Now assume $a<b<c<d$ in $\left\{ 1,\ldots,n\right\} $ with $a,c$
in one block of $\pi\veetild\rho$ and $b,d$ in another. Since $\rho$
is non-crossing, one of these blocks has to be the new block $B\cup B'$
and the other $B''$ -- a third, different block of $\rho$. Without
loss of generality, 
\[
a,c\in\left\{ a_{1}=1,a_{2},\ldots,a_{k},c_{1},\ldots,c_{\ell-1},c_{\ell}=n\right\} 
\]
and $b,d\in B''$. Moreover, we must have $a=a_{i}$ and $c=c_{j}$
for some $i$ and $j$, again because $\rho$ is non-crossing. But
then the four numbers $b<c=c_{j}<d<n$ contradict the fact that $\rho$
is non-crossing.

For a general $\pi\in\intc\left(n\right)$, let $\mathrm{Pairs}\left(\pi\right)=\left\{ \left\{ i,i+1\right\} \,\middle|\,i\in\left\{ 1,\ldots,n\right\} ~\,\&~\,i\stackrel{\pi}{\sim}i+1\right\} $,
where the addition is modulo $n$ so if $i=n$, then $i+1=1$. For
$i\ne j$ in $\left\{ 1,\ldots,n\right\} $, let $\sigma_{i,j}\in\nc\left(n\right)$
be the partition consisting of the block $\left\{ i,j\right\} $ together
with $n-2$ singletons. It is clear that $\pi=\tilde{\bigvee}_{\left\{ i,i+1\right\} \in\mathrm{Pairs}\left(\pi\right)}\sigma_{i,i+1}$,
so $\rho\veetild\pi=\rho\veetild\left(\tilde{\bigvee}_{\left\{ i,i+1\right\} \in\mathrm{Pairs}\left(\pi\right)}\sigma_{i,i+1}\right)$.
Since the operator $\veetild$ is associative, we are done by the
special case. 
\end{proof}
\begin{rem}
As a converse to the preceding proposition, we observe that if $\pi\in\nc(n)$
has the property that $\pi\vee\rho=\pi\veetild\rho$ for every $\rho\in\nc\left(n\right)$,
then it follows that $\pi\in\intc\left(n\right)$. Here is the simple
argument: assume by contradiction that, up to a cyclic shift, there
exists a block $B\in\pi$ and numbers $1\le a<b<c<d\le n$ such that
$a,c\in B$ yet $b,d\notin B$. Consider the non-crossing partition
$\rho\in\nc\left(n\right)$ consisting of the block $\left\{ b,d\right\} $
together with $n-2$ singletons. It is easy to see that $\pi\veetild\rho$
is the partition obtained from $\pi$ by merging the block containing
$b$ with the block containing $d$. Clearly, this partition is not
non-crossing, hence it cannot equal $\pi\vee\rho$. 
\end{rem}

There is a version of Theorem \ref{thm:main-shallow meanders} for
the more generalized notion of interval partitions: 
\begin{thm}
\label{thm:intc-nc}For every $n\in\bN$, the number of pairs $\left(\pi,\rho\right)\in\intc\left(n\right)\times\nc\left(n\right)$
which define a meander and such that $\left|\pi\right|=m$ is 
\begin{equation}
\begin{cases}
1 & m=1\\
\frac{1}{m}\left(\begin{array}{c}
n\\
m-1
\end{array}\right)\left(\begin{array}{c}
n+m-1\\
n-m
\end{array}\right) & m\ge2
\end{cases}.\label{eqn:meanders with intc on top}
\end{equation}
\end{thm}

Compare this with the quantity $\frac{1}{n}\left(\begin{array}{c}
n\\
m-1
\end{array}\right)\left(\begin{array}{c}
n+m-1\\
n-m
\end{array}\right)$ from Theorem \ref{thm:main-shallow meanders}, where $\pi$ is conditioned
to be in the smaller set $\Int\left(n\right)$ rather than in $\intc\left(n\right)$.
The proof of Theorem \ref{thm:intc-nc} follows the same lines as
the one of Theorem \ref{thm:main-shallow meanders}, with minor adaptations,
and we omit it, but do consult Remark \ref{rem:bijection for cyclic interval}
below.

In Remark \ref{rem:exp growth rate of shallow meanders} we mentioned
the exponential growth rate of the number of meanders with shallow
top, and how it is compared with the conjectural exponential growth
rate of the total number of meanders. We now describe the computation
leading from Theorem \ref{thm:main-shallow meanders} to this quantity: 
\begin{cor}
\label{cor:EGR of meanders with shallow top}The exponential growth
rate of the number of meanders with shallow top of order $n$ is roughly
$5.22$. 
\end{cor}

\begin{proof}
Recall that $\mst\left(n\right)$ denotes the set of meanders of order
$n$ with shallow top. The exponential growth rate we compute is,
by definition,$\overline{\lim}_{n\to\infty}\sqrt[n]{\left|\mst\left(n\right)\right|}.$
Since exponential growth rate does not see polynomial factors, for
the sake of computing it we may replace the summation in \eqref{eqn:meanders with shallow top}
by the following quantities: 
\begin{eqnarray*}
\left|\mst\left(n\right)\right| & = & \sum_{m=1}^{n}\frac{1}{n}\binom{n}{m-1}\binom{n+m-1}{n-m}\\
 & \asymp & \max_{1\le m\le n}\binom{n}{m}\binom{n+m}{n-m}\\
 & \asymp & \max_{\alpha\in\left(0,1\right)}\binom{n}{\alpha n}\binom{\left(1+\alpha\right)n}{\left(1-\alpha\right)n},
\end{eqnarray*}
(here $\asymp$ means ``up to polynomial factors in $n$''). By
Stirling's formula, for $\alpha>\beta$ one has $\binom{\alpha n}{\beta n}\asymp\left[\frac{\alpha^{\alpha}}{\beta^{\beta}\left(\alpha-\beta\right)^{\alpha-\beta}}\right]^{n}$
so 
\begin{equation}
\left|\mst\left(n\right)\right|\asymp\left[\max_{\alpha\in\left(0,1\right)}\frac{\left(1+\alpha\right)^{1+\alpha}}{\alpha^{\alpha}\left(1-\alpha\right)^{2\left(1-\alpha\right)}\left(2\alpha\right)^{2\alpha}}\right]^{n}.\label{eq:max-alpha}
\end{equation}
Simple (numerical) analysis shows that the maximum is obtained for
$\alpha\approx0.4694$, for which the fraction in \eqref{eq:max-alpha}
gives, roughly, $5.21914$. 
\end{proof}
\begin{rem}
\begin{enumerate}
\item Interestingly, the value of $m$ which asymptotically generates the
most meanders $M\left(\pi,\rho\right)$ with $\pi\in\Int\left(n\right)$
and $\left|\pi\right|=m$ is, therefore, $m\approx0.47n$. This is
in contrast to the case of general non-crossing partitions, where
it is expected that the maximum is obtained for $m\approx0.5n$. At
least, by the top-down symmetry of $\cH_{n}$, if the maximum is obtained
for $\alpha n$, it is also obtained for $\left(1-\alpha\right)n$.
We also comment that the most common number of blocks in an interval
partition, or in a general non-crossing partition, is $m\approx0.5n$.
See also Conjecture \ref{conj:rainbow partitions} in this regard. 
\item The statement of Corollary \ref{cor:EGR of meanders with shallow top}
holds just as well, and with the same proof, for the somewhat larger
set of meanders with top partition in $\intc\left(n\right)$. 
\end{enumerate}
\end{rem}

\section{{\large{}A Tree Bijection for Meanders with Shallow Top\label{sec:bijection}}}

Recall the graph $\Gamma\left(\pi,\rho\right)$ from Definition \ref{def:Gamma(pi,rho)},
which, by Corollary \ref{cor:equivalence for meander}, is a tree
whenever $M\left(\pi,\rho\right)$ is a meander. When $\pi$ and $\rho$
have order $n$, the graph has $n$ edges labeled by the elements
$\left\{ 1,\ldots,n\right\} $, whence $\pi$ and $\rho$ can be completely
recovered from $\Gamma\left(\pi,\rho\right)$. This section shows
that when $M\left(\pi,\rho\right)$ is a meander with shallow top,
one can ``forget'' the edge-labels of the tree $\Gamma\left(\pi,\rho\right)$,
record, instead, smaller amount of information, and still recover
$\pi$ and $\rho$. Moreover, there is a bijection between trees with
this type of information and the set $\mst=\bigcup_{n=1}^{\infty}\mst\left(n\right)$\marginpar{$\mst$}
of meanders with shallow top.

Recall that the vertices of $\Gamma\left(\pi,\rho\right)$ are colored
black for the blocks of $\pi$ (in the upper half-plane) and white
for the blocks of $\rho$ (in the lower half-plane). It is also a
\emph{fat-tree}\marginpar{fat-tree}, in the sense that it comes with
a particular embedding in the plane, which means that in every vertex
there is a cyclic order on the edges emanating from it (graphs with
this extra information of cyclic orders at every vertex are called
\emph{ribbon graphs}, \emph{fat graphs}, or \emph{cyclic graphs}).
Of course, the cyclic order can be derived from the labels of the
edges: if the edges around a black vertex are labeled $i_{1}<i_{2}<\ldots<i_{k}$,
this is exactly the cyclic order, counterclockwise; if the vertex
is white, this is the cyclic order clockwise.

Now assume that $\pi$ is an interval partition and $M\left(\pi,\rho\right)$
a meander. We then keep the following data from $\Gamma\left(\pi,\rho\right)$: 
\begin{itemize}
\item We record the block of $\pi$ containing the element $1$. We do so
by marking the corresponding black vertex as the root of the tree. 
\item We record the cyclic order at every vertex. 
\item We also record, for every black vertex, the edge with the smallest
label. Namely, there is a special edge for every black vertex. 
\end{itemize}
As for the last piece of data, recall that every black vertex of $\Gamma\left(\pi,\rho\right)$
is a block of $\pi$, so the labels of the edges emanating from it
form an interval in $\left\{ 1,\ldots,n\right\} $. The edge with
the smallest label is, therefore, the beginning point of the interval.

Formally, the resulting tree belongs to the following set of fat-trees: 
\begin{defn}
\label{def:set--of-trees}Let $\fT$\marginpar{$\fT$} denote the
set of finite rooted fat-trees with the following additional data: 
\begin{itemize}
\item The vertices of every $T\in\fT$ are properly colored by black and
white, with the root colored black\footnote{Namely, every edge of the tree has one black endpoint and one white
endpoint. Of course, there is only one way to properly color by black
and white the vertices of a tree with a black root, so there is no
real information here. The coloring is used merely to facilitate the
reference to one of the two subsets of vertices.}. 
\item For every black vertex, one of the edges emanating from it is marked
as special. 
\end{itemize}
The forgetful map\marginpar{$\forget$} 
\[
\forget\colon\mst\to\fT
\]
is defined as above: for $M\left(\pi,\rho\right)\in\mst$, construct
the tree $\Gamma\left(\pi,\rho\right)$ with its embedding in the
plane, mark the black vertex containing $1$ as the root, record the
edge with smallest label at every black vertex, and ``forget'' the
labels of the edges. 
\end{defn}

\begin{thm}
\label{thm:bijection}The map $\forget\colon\mst\to\fT$ is a bijection. 
\end{thm}

In the proof of the theorem we use the following two key lemmas which
capture the crucial property of the tree $\forget\left(M\left(\pi,\rho\right)\right)$
when $\pi$ is an interval partition. In the lemmas, we use the notion
of a \emph{subtree} of a tree $T$ spanned by an oriented edge: if
$\vec{e}=\left(u,v\right)$ is an edge $e$ of $T$ with some orientation,
we let \marginpar{$T_{\vec{e}}$}$T_{\vec{e}}$ denote the subtree
which is the connected component of $v$, the head of $\overrightarrow{e}$,
in $T\setminus\left\{ e\right\} $. Note that $T_{\vec{e}}$ does
not contain the edge $e$ itself. 
\begin{lem}
\label{lem:black vertices}Let $\left(\pi,\rho\right)\in\intc\left(n\right)\times\nc\left(n\right)$
define a meander, so $T:=\Gamma\left(\pi,\rho\right)$ is a tree.
Let $b$ be some black vertex in $T$ with emanating edges, in counterclockwise
order, $\vec{e}_{1},\ldots,\vec{e}_{k}$. Then for $i=1,\ldots,k$,
the edge-labels in the subtree $T_{\vec{e}_{i}}$ form a cyclic interval\footnote{A cyclic interval is a cyclic shift of an interval.}
in $\left\{ 1,\ldots,n\right\} $, and up to cyclic shift, 
\begin{equation}
\left\{ e_{1},\ldots,e_{k}\right\} <E\left(T_{\vec{e}_{k}}\right)<E\left(T_{\vec{e}_{k-1}}\right)<\ldots<E\left(T_{\vec{e}_{1}}\right),\label{eq:order on subsets of edges in a black vertex}
\end{equation}
where $E\left(G\right)$ marks the set of edges of the graph $G$,
two edges are compared according to their labels, and an inequality
between sets $A<B$ means that $a<b$ for every $a\in A,~b\in B$. 
\end{lem}

\begin{proof}
It is enough to prove that if $\vec{e}_{i}$ and $\vec{e}_{i+1}$
emanate from $b$ with labels $j$ and $\left(j+1\right)\mod n$,
respectively, then up to a cyclic shift, 
\begin{equation}
\left\{ e_{1},\ldots,e_{k}\right\} <E\left(T_{\vec{e}_{i+1}}\right)<E\left(T_{\vec{e_{i}}}\right).\label{eq:partial order}
\end{equation}
Indeed, the cyclic order in \eqref{eq:order on subsets of edges in a black vertex}
easily follows from \eqref{eq:partial order}. But the sets of edges
in \eqref{eq:order on subsets of edges in a black vertex} are disjoint
sets which exhaust the entire edges of the tree, so every subset in
\eqref{eq:order on subsets of edges in a black vertex} must form
a cyclic interval.

To prove the claim about $\vec{e}_{i}$ and $\vec{e}_{i+1}$, consider
$\pi'$, the non-crossing partition in $\intc\left(n\right)$ obtained
from $\pi$ by breaking the block $b$ into singletons. The edge-labels
in $T_{\vec{e}_{i}}\cup\left\{ e_{i}\right\} $ and those in $T_{\vec{e}_{i+1}}\cup\left\{ e_{i+1}\right\} $
form two distinct blocks $B_{1}$ and $B_{2}$, respectively, in the
partition $\pi'\veetild\rho$. By Proposition \ref{prop:d_H for interval partitions},
$\pi'\vee\rho=\pi'\veetild\rho$, so, in particular, $\pi'\veetild\rho$
is non-crossing. As $j\in B_{1}$ and $j+1\in B_{2}$, this means
that up to a cyclic shift, 
\[
\left\{ j,j+1\right\} <B_{2}\setminus\left\{ j+1\right\} <B_{1}\setminus\left\{ j\right\} .
\]
Since the labels of $\left\{ e_{1},\ldots,e_{k}\right\} $ form a
cyclic interval containing $\left\{ j,j+1\right\} $, we deduce \eqref{eq:partial order}. 
\end{proof}
\begin{lem}
\label{lem:white vertex}Let $\left(\pi,\rho\right)\in\intc\left(n\right)\times\nc\left(n\right)$
define a meander, so $T:=\Gamma\left(\pi,\rho\right)$ is a tree.
Let $w$ be some black vertex in $T$ with emanating edges, in clockwise
order, $\vec{e}_{1},\ldots,\vec{e}_{k}$. Then for $i=1,\ldots,k$,
the edge-labels of $E\left(T_{\vec{e}_{i}}\right)\cup\left\{ e_{i}\right\} $
form a cyclic interval in $\left\{ 1,\ldots,n\right\} $, and up to
a cyclic shift, 
\begin{equation}
E\left(T_{\vec{e}_{1}}\right)\cup\left\{ e_{1}\right\} <E\left(T_{\vec{e}_{2}}\right)\cup\left\{ e_{2}\right\} <\ldots<E\left(T_{\vec{e}_{k}}\right)\cup\left\{ e_{k}\right\} .\label{eq:order on subtrees from a white vertex}
\end{equation}
\end{lem}

\begin{proof}
For every $\vec{e}$ with white tail, the labels of $E\left(T_{\vec{e}}\right)\cup\left\{ e\right\} $
form a cyclic interval because they are the complement in $\left\{ 1,\ldots,n\right\} $
of the labels of $E\left(T_{\overleftarrow{e}}\right)$ which form
a cyclic interval by Lemma \ref{lem:black vertices}. But by the assumptions,
the cyclic order on $e_{1},\ldots,e_{k}$ means that up to a cyclic
shift $e_{1}<e_{2}<\ldots<e_{k}$, which yields \eqref{eq:order on subtrees from a white vertex}. 
\end{proof}
\begin{proof}[Proof of Theorem \ref{thm:bijection}]
We need to show that $\forget$ is injective and surjective.

\noindent \textbf{Injectivity}

\noindent We start with injectivity, namely, we need to show that
if the pair $\left(\pi,\rho\right)\in\Int\left(n\right)\times\nc\left(n\right)$
corresponds to a meander with shallow top, then we can recover $\pi$
and $\rho$ from $T:=\forget\left(M\left(\pi,\rho\right)\right)$.
What we need to recover is the labels of the edges. Of course, $n$
is the number of edges of $T$. The edges emanating from every black
vertex form an interval, and the order inside the interval is known
thanks to the marked edge in every black vertex in $T$. Hence, it
is enough to recover the order in which the different intervals lie
in $\left\{ 1,\ldots,n\right\} $, namely, to recover the order induced
on the black vertices. We let $B\left(T\right)$ denote the set of
black vertices, and $\prec$ denote the order on $B\left(T\right)$,
so $\left(B\left(T\right),\prec\right)$ is an ordered set.

Lemmas \ref{lem:black vertices} and \ref{lem:white vertex} yield
that for every oriented edge $\vec{e}$, be its head black or white,
the set $B\left(T_{\vec{e}}\right)$ of black vertices in $T_{\vec{e}}$
form a cyclic interval in $\left(B\left(T\right),\prec\right)$. Moreover,
if $\vec{e}$ points away from the root of $T$, then $B\left(T_{\vec{e}}\right)$
forms an actual interval (rather than a cyclic interval), as $B\left(T_{\vec{e}}\right)$
does not contain the smallest black vertex which is the root.

Thus, to fully recover the order on the black vertices of $T$, it
is enough to determine the following: 
\begin{itemize}
\item for every black vertex $b$ with emanating edges away from the root
$\vec{e}_{1},\ldots,\vec{e}_{k}$, the relative order among $\left\{ b\right\} ,B\left(T_{\vec{e}_{1}}\right),\ldots,B\left(T_{\vec{e}_{k}}\right)$ 
\item for every white vertex $w$ with emanating edges away from the root
$\vec{e}_{1},\ldots,\vec{e}_{k}$, the relative order among $B\left(T_{\vec{e}_{1}}\right),\ldots,B\left(T_{\vec{e}_{k}}\right)$ 
\end{itemize}
The above lemmas yield the following solution to this task: 
\begin{itemize}
\item If $b$ is the root of $T$, and the edges emanating from it in order
are $\vec{e}_{1},\ldots,\vec{e}_{k}$, then by Lemma \ref{lem:black vertices},
the required order is 
\[
\left\{ b\right\} \prec B\left(T_{\vec{e}_{k}}\right)\prec B\left(T_{\vec{e}_{k-1}}\right)\prec\ldots\prec B\left(T_{\vec{e}_{1}}\right).
\]
\item If $b$ is a non-root black vertex of $T$ and the edges emanating
from it in the order of the interval $b$ represents are $\vec{e}_{1},\ldots,\vec{e}_{k}$,
with $\vec{e}_{j}$ pointing to the root, the required order is 
\[
B\left(T_{\vec{e}_{j-1}}\right)\prec B\left(T_{\vec{e}_{j-2}}\right)\prec\ldots\prec B\left(T_{\vec{e}_{1}}\right)\prec\left\{ b\right\} \prec B\left(T_{\vec{e}_{k}}\right)\prec B\left(T_{\vec{e}_{k-1}}\right)\prec\ldots\prec B\left(T_{\vec{e}_{j+1}}\right),
\]
because Lemma \ref{lem:black vertices} gives the order up to a cyclic
shift and $B\left(T_{\vec{e}_{j}}\right)$ contains the root which
is the smallest black vertex. 
\item If $w$ is a white vertex of $T$ and the edges emanating from it
clockwise cyclic order are $\vec{e}_{1},\ldots,\vec{e}_{k}$, with
$\vec{e}_{1}$ pointing to the root, the required order is 
\[
B\left(T_{\vec{e}_{2}}\right)\prec B\left(T_{\vec{e}_{3}}\right)\prec\ldots\prec B\left(T_{\vec{e}_{k}}\right),
\]
because Lemma \ref{lem:white vertex} translates to the cyclic order
$B\left(T_{\vec{e}_{1}}\right)\prec B\left(T_{\vec{e}_{2}}\right)\prec\ldots\prec B\left(T_{\vec{e}_{k}}\right)$,
and again $B\left(T_{\vec{e}_{1}}\right)$ contains the root which
is the smallest black vertex. 
\end{itemize}
We illustrate this procedure of recovering the order $\Gamma\left(\pi,\rho\right)$
in Figure \ref{fig:bijection}.

\noindent \textbf{Surjectivity}

\noindent To prove surjectivity of $\forget$, take an arbitrary $T\in\fT$.
Let $n$ denote the number of edges of $T$. We need to show one can
construct $\left(\pi,\rho\right)\in\Int\left(n\right)\times\nc\left(n\right)$
so that $M\left(\pi,\rho\right)\in\mst\left(n\right)$ and $\forget\left(M\left(\pi,\rho\right)\right)=T$.
One can label the edges of $T$ by $\left\{ 1,\ldots,n\right\} $
(in a bijection) by following the ``recovery procedure'' above,
obtaining $T_{\mathrm{colored}}$. The graph $T_{\mathrm{colored}}$
looks like a $\Gamma\left(\pi,\rho\right)$ in terms of the data it
holds, but we still need to show it is actually equal to some $\Gamma\left(\pi,\rho\right)$.

The unique possible candidate for $\pi$ (resp.~$\rho$) is obvious:
this is the partition of $\left\{ 1,\ldots,n\right\} $ defined by
the black (resp.~white) vertices of $T_{\mathrm{colored}}$. It is
clear by the procedure that $\pi\in\Int\left(n\right)$. To see that
$\rho\in\nc\left(n\right)$, let $w_{1}$ and $w_{2}$ be two white
blocks (vertices). We show that $w_{1}$ and $w_{2}$ do not cross.
We separate into three cases: 
\begin{itemize}
\item If $w_{1}$ and $w_{2}$ are both descendants in $T$ of the black
vertex $b$, but one is in $T_{\vec{e}_{1}}$ and one in $T_{\vec{e}_{2}}$
for two distinct edges emanating from $b$, then the edges of $w_{1}$
(resp.~$w_{2}$) are contained in $T_{\vec{e}_{1}}\cup\left\{ e_{1}\right\} $
(resp.~$T_{\vec{e}_{2}}\cup\left\{ e_{2}\right\} $), and the recovery
procedure ensures that the sets $T_{\vec{e}_{1}}\cup\left\{ e_{1}\right\} $
and $T_{\vec{e}_{2}}\cup\left\{ e_{2}\right\} $ lie in two distinct
cyclic intervals, whence they do not cross. 
\item If $w_{1}$ and $w_{2}$ are both proper descendants in $T$ of the
white vertex $w$ but, one is in $T_{\vec{e}_{1}}$ and one in $T_{\vec{e}_{2}}$
for two distinct edges emanating from $w$, then the recovery procedure
ensures that $w_{1}$ and $w_{2}$ are contained in two distinct intervals,
whence they do not cross. 
\item If one of $w_{1},w_{2}$ is a descendant of the other, say without
loss of generality, that $w_{2}$ is a descendant of $w_{1}$, consider
$\vec{e}$, the last edge in the geodesic from $w_{1}$ to $w_{2}$.
The recovery procedure ensures that $\left\{ e\right\} \cup E\left(T_{\vec{e}}\right)$,
which contains $w_{2}$, is contained in an interval disjoint from
$w_{1}$. 
\end{itemize}
Hence $\left(\pi,\rho\right)\in\Int\left(n\right)\times\nc\left(n\right)$.
Finally, it is easy to check that the procedure ensures that the cyclic
order on the edges in every vertex of $T$ matches the cyclic order
induced by the labels in $T_{\mathrm{colors}}$. Hence $\Gamma\left(\pi,\rho\right)$
is the exact same fat-tree (tree with an embedding in the plane) as
$T_{color}$. Thus $M\left(\pi,\rho\right)\in\mst\left(n\right)$
and $\forget\left(M\left(\pi,\rho\right)\right)=T$. 
\end{proof}
\begin{figure}
\centering{}\includegraphics[viewport=170bp 380bp 370bp 500bp,scale=1.05]{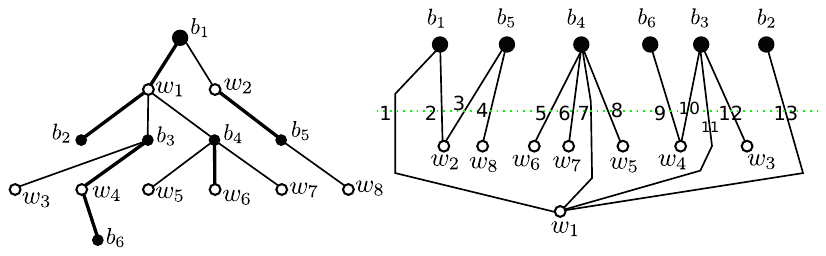}\caption{\label{fig:bijection}The left part is a tree $T\in\fT$ with root
$b_{1}$ and a special edge for every black vertex marked by a thicker
line. The recovery procedure from the proof of Theorem \ref{thm:bijection}
gives, in this case, the following order on the black vertices: by
the root $b_{1}$, we obtain that $\left\{ b_{1}\right\} <\left\{ b_{5}\right\} <\left\{ b_{2},b_{3},b_{4},b_{6}\right\} $,
in $w_{1}$ we get $\left\{ b_{4}\right\} <\left\{ b_{3},b_{6}\right\} <\left\{ b_{2}\right\} $,
and in $b_{3}$ we see that $\left\{ b_{6}\right\} <\left\{ b_{3}\right\} $.
The recovered meander $M\left(\pi,\rho\right)\in\mst\left(13\right)$,
in the shape of the tree $\Gamma\left(\pi,\rho\right)$, is on the
right. We stress that the vertex labels are not part of the data of
$T$ nor of $\Gamma\left(\pi,\rho\right)$: they are given here merely
to clarify the right isomorphism between the two fat-trees.}
\end{figure}

\begin{rem}
\label{rem:bijection for cyclic interval}There is an analog of Theorem
\ref{thm:bijection} for meanders in $\intc\left(n\right)\times\nc\left(n\right)$.
This time, the element $1$ is not always the first element in its
cyclic interval in $\pi$, so we also record the edge of the root
corresponding to $1$. So now the root of every tree in $\fT$ should
have two marked edges: one marked as $1$, and one, possibly the same
one, marked as the smallest in the interval. The proof of the bijection
is almost the same, with small adaptation to the change in the root. 
\end{rem}

\section{{\large{}Enumerative Consequences of the Tree Bijection \label{sec:Enumerative-consequences}}}

In this section, we determine the generating function for the set
of fat-trees $\fT$ defined in Definition \ref{def:set--of-trees}
and deduce the enumerative results of Theorem \ref{thm:main-shallow meanders}.
In order to do so, we define two additional sets of bicolored fat-trees
with additional data. These sets describe proper subtrees of trees
in $\fT$, where a proper subtree here means the subtree of some $T\in\fT$
spanned by a non-root vertex, its \textbf{parent} and all its descendants.
We let \marginpar{$\fT_{w},\fT_{b}$}$\fT_{w}$ ($\fT_{b}$, respectively)
denote the set of possible proper fat-subtrees spanned by a \emph{white
}(respectively, \emph{black}) vertex, its parent and its descendants.
So every $T\in\fT_{w}$ (resp.~$T\in\fT_{b}$) is a fat-tree with
some white (resp.~black) vertex marked as the root and some black
(resp.~white) neighbor of the root, which is a leaf of a tree, marked
as the parent of the root. In addition, every black vertex has one
emanating edge marked as special.

We let $\Phi$, $W$ and $B$ denote the generating functions of the
sets $\fT$, $\fT_{w}$ and $\fT_{b}$, respectively. In the following
generating functions, the exponent of the variable $x$ is the number
of edges (which is $n$), the exponent of the variable $y$ is the
number of black vertices (which is the number of blocks in $\pi\in\Int\left(n\right)$),
the exponent of $z_{k}$ is the number of black vertices of degree
$k$, and the exponent of $w_{k}$ is the number of white vertices
of degree $k$. To obtain this, we think of every black vertex of
degree $k$ as marked by $yx^{k}z_{k}$ and every white vertex of
degree $k$ as marked by $w_{k}$.

Since the root of a tree $T\in\fT_{w}$ has some degree $\ell\ge1$
and $\left(\ell-1\right)$ black children, we obtain 
\begin{equation}
W=\sum_{\ell\ge1}w_{\ell}B^{\ell-1}.\label{for W}
\end{equation}
Analogously, the root of a tree $T\in\fT_{b}$ has some degree $k\ge1$
with $\left(k-1\right)$ white children and $k$ choices for which
edge is special, so

\begin{equation}
B=\sum_{k\ge1}kyx^{k}z_{k}W^{k-1}.\label{for B}
\end{equation}

Similarly, the root of a tree $T\in\fT$ has degree $k$ for some
$k\ge1$, with $k$ white children. Unlike the case of a subtree rooted
at a black vertex, here there is no special edge at the root pointing
to the parent, so all edges at the root look the same and there is
no real choice of which one to mark as special. Thus, we obtain the
following equation for $\Phi$ in terms of $W$: 
\begin{equation}
\Phi=\sum_{k\ge1}yx^{k}z_{k}W^{k}.\label{forPhi}
\end{equation}
Now, substituting for $W$ via~(\ref{for W}) in (\ref{for B}) and
(\ref{forPhi}), we get the pair of functional equations 
\begin{equation}
\Phi=\sum_{k\ge1}yx^{k}z_{k}\left(\sum_{\ell\ge1}w_{\ell}B^{\ell-1}\right)^{k},\qquad\qquad B=y\sum_{k\ge1}kx^{k}z_{k}\left(\sum_{\ell\ge1}w_{\ell}B^{\ell-1}\right)^{k-1}.\label{lagsyst}
\end{equation}
Also, using the notation \marginpar{$\bi,\bj$}$\bi=(i_{1},i_{2},\ldots)$
and $\bj=(j_{1},j_{2},\ldots)$, let \marginpar{$N(m,n;\bi,\bj)$}$N(m,n;\bi,\bj)$
denote the number of meandric pairs $(\pi,\rho)$ in which the interval
partition $\pi\in\Int(n)$ has $m$ blocks, of which $j_{k}$ blocks
have size $k$, $k\ge1$, and $\rho\in\nc(n)$ has $n-m+1$ blocks,
of which $i_{\ell}$ blocks have size $\ell$, $\ell\ge1$. Of course,
the elements of $\bi$ and $\bj$ are subject to the restrictions:
\begin{equation}
\sum_{\ell\ge1}i_{\ell}=n-m+1,\qquad\sum_{\ell\ge1}\ell i_{\ell}=n,\qquad\sum_{k\ge1}j_{k}=m,\qquad\sum_{k\ge1}kj_{k}=n.\label{eq:ijrestns}
\end{equation}
Let \marginpar{$N_{a}(m,n;\bi,\bj)$}$N_{a}(m,n;\bi,\bj)$ denote
the further restricted number of meandric pairs above such that the
block of $\pi$ containing the element $1$ is of size $a$ (so it
is the block $\{1,\ldots,a\}$), where we have $j_{a}\ge1$.

We now determine explicit formulas for the numbers $N(m,n;\bi,\bj)$
and $N_{a}(m,n;\bi,\bj)$ by using the bijection from Theorem \ref{thm:bijection}. 
\begin{prop}
\label{prop:N(m,n,i,j)} For $n\ge m\ge1$, and $\bi,\bj$ satisfying
(\ref{eq:ijrestns}), we have 
\begin{equation}
N(m,n;\bi,\bj)=m!(n-m)!\prod_{\ell\ge1}\frac{1}{i_{\ell}!}\prod_{k\ge1}\frac{k^{j_{k}}}{j_{k}!}\label{eq:Nformula}
\end{equation}
and (where $j_{a}\ge1$), 
\begin{equation}
N_{a}(m,n;\bi,\bj)=j_{a}(m-1)!(n-m)!\prod_{\ell\ge1}\frac{1}{i_{\ell}!}\prod_{k\ge1}\frac{k^{j_{k}}}{j_{k}!}\label{eq:Naformula}
\end{equation}
\end{prop}

\begin{proof}
For the first part of the result, if $m=1$, the restrictions \eqref{eq:ijrestns}
yield that $\bj=\left(0,\ldots0,1,0,\ldots\right)$ with $j_{n}=1$,
and that $\bi=\left(n,0,0,\ldots\right)$. Then the formula \eqref{eq:Nformula}
gives $1$ which is easily seen to be equal to $N\left(1,n;\bi,\bj\right)$.
Now assume $m\ge2$. From the bijection from Theorem \ref{thm:bijection},
$N(m,n;\bi,\bj)$ is equal to the number of fat-trees in $\fT$ with
$i_{\ell}$ white vertices of degree $\ell$ and $j_{\ell}$ black
vertices of degree $\ell$, $\ell\ge1$. Thus, using the notation
$\bw^{\bi}=w_{1}^{i_{1}}w_{2}^{i_{2}}\cdots$ and $\bz^{\bj}=z_{1}^{j_{1}}z_{2}^{j_{2}}\cdots$,
we conclude that 
\[
N(m,n;\bi,\bj)=[y^{m}x^{n}\bw^{\bi}\bz^{\bj}]\Phi,
\]
where we also use here the notation $[F]G$ to denote the \emph{coefficient}
of the monomial $F$ in the expansion of the formal power series $G$.
This is because $\Phi$ has been defined to be precisely the appropriate
generating function.

Now we can determine $N:=N\left(m,n;\bi,\bj\right)$ by applying Lagrange's
Implicit Function Theorem to solve the equations in \eqref{lagsyst}
(see e.g.~\cite[Theorem 1.2.4(1), Page 17]{goulden1983combinatorial})
via the following calculation. Recall that $m\ge2$: 
\begin{eqnarray}
N & = & \left[y^{m}x^{n}\bw^{\bi}\bz^{\bj}\right]\sum_{k\ge1}yx^{k}z_{k}\left(\sum_{\ell\ge1}w_{\ell}B^{\ell-1}\right)^{k}=\left[y^{m-1}x^{n}\bw^{\bi}\bz^{\bj}\right]\sum_{k\ge1}x^{k}z_{k}\left(\sum_{\ell\ge1}w_{\ell}B^{\ell-1}\right)^{k}\nonumber \\
 & = & \frac{1}{m-1}\left[\la^{m-2}x^{n}\bw^{\bi}\bz^{\bj}\right]\frac{d}{d\lambda}\left[\sum_{k\ge1}x^{k}z_{k}\left(\sum_{\ell\ge1}w_{\ell}\la^{\ell-1}\right)^{k}\right]\left(\sum_{k\ge1}kx^{k}z_{k}\left(\sum_{\ell\ge1}w_{\ell}\lambda{}^{\ell-1}\right)^{k-1}\right)^{m-1}\nonumber \\
 & = & \frac{1}{m-1}\left[\la^{m-2}x^{n}\bw^{\bi}\bz^{\bj}\right]\frac{d}{d\lambda}\left(\sum_{\ell\ge1}w_{\ell}\lambda^{\ell-1}\right)\left(\sum_{k\ge1}kx^{k}z_{k}\left(\sum_{\ell\ge1}w_{\ell}\lambda{}^{\ell-1}\right)^{k-1}\right)^{m}.\label{eq:step I for N}
\end{eqnarray}
We now compute the coefficient of $\left[x^{n}\bz^{\bj}\right]$ in
\eqref{eq:step I for N}. There are exactly $\binom{m}{j_{1},j_{2},\ldots}=\frac{m!}{j_{1}!j_{2}!\ldots}$
ways to pick exactly $j_{r}$ times the summand containing $z_{r}$
for every $r\ge1$, and then the coefficient of $x$ is exactly $\sum_{r\ge1}rj_{r}=n$,
as required. So 
\begin{eqnarray*}
N & = & \frac{m!}{m-1}\prod_{k\ge1}\frac{k^{j_{k}}}{j_{k}!}\left[\lambda^{m-2}\bw^{\bi}\right]\frac{d}{d\lambda}\left(\sum_{\ell\ge1}w_{\ell}\lambda^{\ell-1}\right)\left(\sum_{\ell\ge1}w_{\ell}\lambda{}^{\ell-1}\right)^{n-m}\\
 & = & \frac{m!}{m-1}\prod_{k\ge1}\frac{k^{j_{k}}}{j_{k}!}\left[\lambda^{m-2}\bw^{\bi}\right]\frac{1}{n-m+1}\cdot\frac{d}{d\lambda}\left(\sum_{\ell\ge1}w_{\ell}\lambda^{\ell-1}\right)^{n-m+1}\\
 & \stackrel{\left(*\right)}{=} & \frac{m!}{n-m+1}\prod_{k\ge1}\frac{k^{j_{k}}}{j_{k}!}\left[\lambda^{m-1}\bw^{\bi}\right]\left(\sum_{\ell\ge1}w_{\ell}\la^{\ell-1}\right)^{n-m+1}\\
 & = & \frac{m!}{n-m+1}\prod_{k\ge1}\frac{k^{j_{k}}}{j_{k}!}\binom{n-m+1}{i_{1},i_{2},\ldots}=m!\left(n-m\right)!\prod_{k\ge1}\frac{k^{j_{k}}}{j_{k}!}\prod_{\ell\ge1}\frac{1}{i_{\ell}!}.
\end{eqnarray*}
where the equality $\stackrel{\left(*\right)}{=}$ follows from the
fact that $[\la^{m-2}]\frac{d}{d\la}f(\la)=(m-1)[\la^{m-1}]f(\la)$
for $m\ge2$ and any formal power series $f$. This gives \eqref{eq:Nformula}.

For the second part of the result, it follows immediately from cyclic
symmetry that $N_{a}\left(m,n;\bi,\bj\right)=\frac{j_{a}}{m}N\left(m,n;\bi,\bj\right)$,
and thus \eqref{eq:Naformula} follows from \eqref{eq:Nformula}. 
\end{proof}
Now suppose that we are not interested in the distribution of block
sizes in $\rho$, and thus let $N(m,n;\bj)$ denote the number of
meandric pairs $(\pi,\rho)$ in which the interval partition $\pi\in\Int(n)$
has $m$ blocks, of which $j_{k}$ blocks have size $k$, $1\leq k\leq n$,
and $\rho\in\nc(n)$ has $n-m+1$ blocks, subject to the restrictions:
\begin{equation}
\sum_{k=1}^{n}j_{k}=m,\qquad\sum_{k=1}^{n}kj_{k}=n.\label{eq:jrestns}
\end{equation}
We also let $N_{a}(m,n;\bj)$ denote the further restricted number
of meandric pairs above such that the block of $\pi$ containing the
element $1$ is of size $a$ (so it is the block $\{1,\ldots,a\}$),
where we have $j_{a}\ge1$.

We now determine explicit formulas for the numbers $N(m,n;\bj)$ and
$N_{a}(m,n;\bj)$ by summing the results of Proposition \ref{prop:N(m,n,i,j)}.
\begin{prop}
\label{prop:N(m,n,j)} For $n\ge m\ge1$, and $\bj$ satisfying \eqref{eq:jrestns},
we have 
\begin{equation}
N(m,n;\bj)=\frac{m(n-1)!}{(n-m+1)!}\prod_{k\ge1}\frac{k^{j_{k}}}{j_{k}!}\label{eq:Njformula}
\end{equation}
and (where $j_{a}\ge1$), 
\begin{equation}
N_{a}(m,n;\bj)=\frac{j_{a}(n-1)!}{(n-m+1)!}\prod_{k\ge1}\frac{k^{j_{k}}}{j_{k}!}\label{eq:Najformula}
\end{equation}
\end{prop}

\begin{proof}
We have 
\[
N\left(m,n;\bj\right)=\!\!\!\!\sum_{\substack{i_{1},i_{2},\ldots\ge0\\
i_{1}+i_{2}+\ldots=n-m+1\\
i_{1}+2i_{2}+\ldots=n
}
}\!\!\!\!\!\!\!\!N(m,n;\bi,\bj)=m!(n-m)!\prod_{k\ge1}\frac{k^{j_{k}}}{j_{k}!}\sum_{\substack{i_{1},i_{2},\ldots\ge0\\
i_{1}+i_{2}+\ldots=n-m+1\\
i_{1}+2i_{2}+\ldots=n
}
}\prod_{\ell\ge1}\frac{1}{i_{\ell!}},
\]
from \eqref{eq:Nformula}. But elementary generating function methods
yield 
\begin{align*}
\sum_{\substack{i_{1},i_{2},\ldots\ge0\\
i_{1}+i_{2}+\ldots=n-m+1\\
i_{1}+2i_{2}+\ldots=n
}
}\prod_{\ell\ge1}\frac{1}{i_{\ell}!} & =\frac{1}{(n-m+1)!}\left[x^{n}\right]\left(\sum_{\ell\ge1}x^{\ell}\right)^{n-m+1}\!\!\!\!\!\!\!\!\!\!\!\!=\frac{1}{(n-m+1)!}\left[x^{n}\right]\left(x(1-x)^{-1}\right)^{n-m+1}\\
 & =\frac{1}{(n-m+1)!}{n-1 \choose m-1},
\end{align*}
and the first part of the result \eqref{eq:Njformula} follows immediately.

The second part of the result follows from the first part, together
with the fact that $N_{a}(m,n;\bj)=\frac{j_{a}}{m}N(m,n;\bj)$. 
\end{proof}
Proposition \ref{prop:16} is a special case of Proposition \ref{prop:N(m,n,j)}:
\begin{proof}[Proof of Proposition \ref{prop:16} ]
The partition $\lambda_{\ell,m}$ is the only interval partition
in $\Int\left(\ell m\right)$ with $m$ blocks of size $\ell$ each.
Thus, the number of meandric partners of $\lambda_{\ell,m}$ is equal
to $N\left(m,\ell m,\bj\right)$, where $j_{\ell}=m$ and $j_{k}=0$
for every $k\neq\ell$. Upon substituting these values on the right-hand
side of Equation (\ref{eq:Njformula}) and suitably regrouping the
factors, we find the numbers announced in Proposition \ref{prop:16}. 
\end{proof}
Finally, suppose that in addition we are not interested in the distribution
of block sizes in $\pi$, and thus let $N(m,n)$ denote the number
of meandric pairs $(\pi,\rho)$ in which the interval partition $\pi\in\nc(n)$
has $m$ blocks (where $n\ge m\ge1$) and $\rho\in\nc(n)$ has $n-m+1$
blocks. We now determine explicit formulas for the numbers $N(m,n)$
by summing the first part of Proposition \ref{prop:N(m,n,j)}. This
is exactly the content of Theorem \ref{thm:main-shallow meanders}:
\begin{proof}[Proof of Theorem \ref{thm:main-shallow meanders} ]
We have 
\[
N(m,n)=\!\!\!\!\sum_{\substack{j_{1},j_{2},\ldots\ge0\\
j_{1}+j_{2}+\ldots=m\\
j_{1}+2j_{2}+\ldots=n
}
}\!\!\!\!\!\!\!\!N(m,n;\bj)=\frac{m(n-1)!}{(n-m+1)!}\!\!\!\!\sum_{\substack{j_{1},j_{2},\ldots\ge0\\
j_{1}+j_{2}+\ldots=m\\
j_{1}+2j_{2}+\ldots=n
}
}\prod_{k\ge1}\frac{k^{j_{k}}}{j_{k}!}
\]
from \eqref{eq:Njformula}. But elementary generating function methods
yield 
\begin{align*}
\sum_{\substack{j_{1},j_{2},\ldots\ge0\\
j_{1}+j_{2}+\ldots=m\\
j_{1}+2j_{2}+\ldots=n
}
}\prod_{k\ge1}\frac{k^{j_{k}}}{j_{k}!} & =\frac{1}{m!}\left[x^{n}\right]\left(\sum_{k\ge1}kx^{k}\right)^{m}=\frac{1}{m!}\left[x^{n}\right]\left(x(1-x)^{-2}\right)^{m}=\frac{1}{m!}{n+m-1 \choose n-m},
\end{align*}
and the result follows immediately. 
\begin{rem}
\label{rem:shorter-proof-for-thm-1.1}The anonymous referee drew our
attention to an alternative, shorter proof for Theorems \ref{thm:main-shallow meanders}
and \ref{thm:intc-nc}. For $m\ge2$, a partition $\pi\in\nc\left(n\right)$
is a cyclic permutation of an interval partition (so $\pi\in\intc\left(n\right)$)
with $m$ blocks if and only if its Kreweras complement has one block
of size $m$ and $n-m$ singletons. So the number of pairs $\left(\pi,\rho\right)\in\intc\left(n\right)\times\nc\left(n\right)$
defining a meander with $\left|\pi\right|=m$ is equal to the number
of pairs $\left(\tau,\theta\right)\in\nc\left(n\right)^{2}$ defining
a meander with $\tau$ having a single block of size $m$ and $n-m$
singletons. In this case, $\theta$ must have $m$ blocks. Given $\theta$,
the $m$-block of $\tau$ must contain exactly one point from every
block of $\theta$ (by \ref{enu:meander}$\Leftrightarrow$\ref{enu:tree}
in Corollary \ref{cor:equivalence for meander}), so every $\theta\in\nc\left(n\right)$
with $m$ blocks has exactly $\prod_{V\in\theta}\left|V\right|$ options
for $\tau$. Hence, if $F\left(z,t\right)$ is the generating function
for this problem with $\left[z^{n}t^{m}\right]F\left(z,t\right)$
being the number we seek, it satisfies
\[
F\left(z,t\right)=\sum_{n\ge1}\sum_{\theta\in\nc\left(n\right)}\prod_{V\in\theta}\left(t\left|V\right|\right).
\]
By \cite[Theorem 10.23]{nica-speicher2006lectures}, this equation
is equivalent to 
\[
F=tG\left(z\left(1+F\right)\right)~~\mathrm{with}~~G\left(z\right)=\frac{z}{\left(1-z\right)^{2}}
\]
(as in Equations \eqref{ftog} and \eqref{FtoG} below, with $g_{k}=k$).
This equation can be solved using Lagrange's Implicit Function Theorem
(e.g.~\cite[Theorem 1.2.4(i)]{goulden1983combinatorial}) to obtain
the formula \eqref{eqn:meanders with intc on top} in Theorem \ref{thm:intc-nc}.
Theorem \ref{thm:main-shallow meanders} can then be easily deduced
from Theorem \ref{thm:intc-nc}: for every $\pi\in\intc\left(n\right)$
with $\left|\pi\right|=m\ge2$, exactly $\frac{m}{n}$ of its different
cyclic shifts constitute an interval partition.

This proof makes the analogy with the formula for Narayana numbers
more transparent since these numbers are computed using $G\left(z\right)=\frac{z}{1-z}$.
Of course, this approach does not give the more general result of
Proposition \ref{prop:N(m,n,i,j)} and its corollaries such as Proposition
\ref{prop:16}.
\end{rem}

\end{proof}

\section{{\large{}Distance Distributions in the Graphs $\cH_{n}$ \label{sec:Distance-distributions}}}

In this section we study average distances in $\cH_{n}$. For this
sake, we define the number \marginpar{$b\left(\pi,\rho\right)$}$b(\pi,\rho)$
for two non-crossing partitions $\pi,\rho\in\nc\left(n\right)$, by
\begin{equation}
b(\pi,\rho)=\vert\pi\vert+\vert\rho\vert-2\left|\pi\vee\rho\right|.\label{distdef}
\end{equation}
Recall from Theorem \ref{thm:equivalences for d_H}\eqref{enu:|pi|+|rho|-2pi_0}
that $d_{H}\left(\pi,\rho\right)=\vert\pi\vert+\vert\rho\vert-2\left|\pi\veetild\rho\right|$,
and as $\left|\pi\vee\rho\right|\le\left|\pi\veetild\rho\right|$,
the quantity $b\left(\pi,\rho\right)$ is in general an \emph{upper
bound} on the distance. However, when $\pi\in\intc\left(n\right)$,
we have the equality $b\left(\pi,\rho\right)=d_{H}\left(\pi,\rho\right)$
by Proposition \ref{prop:d_H for interval partitions}. We consider
the average value of $b$ for three families: 
\begin{itemize}
\item Let $\widetilde{d}_{2,m}$, $m\ge1$, denote the average value of
$b(\lambda_{2,m},\rho)=d_{H}\left(\lambda_{2,m},\rho\right)$ where\linebreak{}
$\la_{2,m}=\{\{1,2\},\{3,4\},\ldots,\{2m-1,2m\}\}\in\Int\left(2m\right)$
is fixed and $\rho\in\NC(2m)$, as in Proposition \ref{prop:17}. 
\item Let $d_{n}$, $n\ge1$, denote the average value of $b(\pi,\rho)=d_{H}\left(\pi,\rho\right)$
where $\pi\in\Int(n)$ and $\rho\in\nc(n)$, as in Theorem \ref{thm:avg distance from interval partition}. 
\item Let \marginpar{$\widetilde{b}_{n}$}$\widetilde{b}_{n}$, $n\ge1$,
denote the average value of $b(\pi,\rho)$ where $\pi,\rho\in\nc\left(n\right)$.
So $\widetilde{b}_{n}$ is an upper bound on the average distance
in $\cH_{n}$. 
\end{itemize}
Now, to determine these average values, we begin with the simplest
contributions. The number of non-crossing partitions $\rho\in\nc\left(n\right)$
with $\left|\rho\right|=k$ blocks is given by the Narayana number
$\frac{1}{n}\binom{n}{k}{n \choose k-1}$ (e.g.~\cite[Corollaire 4.1]{kreweras1972partitions}),
and the number of interval partitions $\pi\in\Int(n)$ with $k$ blocks
is given by ${n-1 \choose k-1}$. In both cases, these distributions
are symmetric about the centre $\tfrac{1}{2}(n+1)$, which implies
that the average values both for $\left|\rho\right|$, $\rho\in\nc\left(n\right)$
and for $\vert\pi\vert$, $\pi\in\Int(n)$, are given by 
\begin{equation}
\tfrac{1}{2}(n+1).\label{avgeblocks}
\end{equation}
To evaluate the averages for the three families above, we also need
to consider the more complicated $\vert\pi\vee\rho\vert$ term in
\eqref{distdef}. To help with this term, we introduce the three generating
functions 
\begin{align*}
\Psi^{(1)}(z,t) & =\sum_{m\ge1}z^{m}\sum_{\rho\in\nc(2m)}t^{\left|\lambda_{2,m}\vee\rho\right|},\\
\Psi^{(2)}\left(z,t\right) & =\sum_{n\ge1}y^{n}\sum_{\substack{\pi\in\Int\left(n\right)\\
\rho\in\nc(n)
}
}z^{\vert\pi\vert}t^{\vert\pi\vee\rho\vert},\\
\Psi^{(3)}(z,t) & =\sum_{n\ge1}z^{n}\sum_{\pi,\rho\in\nc(n)}t^{\vert\pi\vee\rho\vert},
\end{align*}
and note that from elementary considerations we have 
\begin{equation}
\Psi^{(1)}(z,1)=\sum_{n\ge1}\Cat_{2m}z^{m},\;\;\Psi^{(2)}(z,1)=\sum_{n\ge1}\Cat_{n}z(1+z)^{n-1}y^{n},\;\;\Psi^{(3)}(z,1)=\sum_{n\ge1}\Cat_{n}^{2}z^{n}.\label{seriesone}
\end{equation}
Note that a third variable $y$ appears in $\Psi^{(2)}(z,t)$: we
consider $\Psi^{(2)}(z,t)$ to be a formal power series in the variables
$z$ and $t$, with coefficients that are formal power series in the
variable $y$.

Now in order to evaluate the numerator for the contribution of the
$\vert\pi\vee\rho\vert$ term in the averages $\widetilde{d}_{2,m}$,
$d_{n}$ and $\widetilde{b}_{n}$, we need to apply the partial derivative
$\frac{\partial}{\partial t}$ to $\Psi^{(i)}$, $i=1,2,3$, and set
$t=1$. The following proposition will allow us to carry out this
process in the three cases, by appropriate specialization. A comment
related to the result of Proposition \ref{pr:101} is that it has
a nice interpretation in terms of the operation of ``free additive
convolution'' used in free probability; this is explained in Remark
\ref{rem:82}, following the proposition.

A notational detail: in Proposition \ref{pr:101} (and later on in
this section) we use subscripts to denote partial derivatives, via
the notation \marginpar{$F_{i}$}$F_{i}=\frac{\partial}{\partial x_{i}}F$,
$i=1,\ldots,m$, for a formal power series $F(x_{1},\ldots,x_{m})$
in $m\ge1$ variables. 
\begin{prop}
\label{pr:101}Let 
\begin{equation}
F(z,t)=\sum_{n\ge1}z^{n}\sum_{\theta\in NC(n)}\;\;\prod_{\theta_{i}\in\theta}\Big(t\,g_{|\theta_{i}|}\Big),\label{ftog}
\end{equation}
where $g_{1},g_{2},\ldots$ are expressions which do not depend on
$z$ nor on $t$. Then 
\[
F_{2}(z,1)=F(z,1)+zF_{1}(z,1)-\frac{zF_{1}(z,1)}{1+F(z,1)}.
\]
\end{prop}

\begin{proof}
By \cite[Theorem 10.23]{nica-speicher2006lectures}, Equation \eqref{ftog}
is equivalent to the functional equation 
\begin{equation}
F(z,t)=t\,G\left(z\left(1+F\left(z,t\right)\right)\right),\label{FtoG}
\end{equation}
where $G(z)=\sum_{m\ge1}g_{m}z^{m}$. Now, applying $\frac{\partial}{\partial t}$
to \eqref{FtoG} and setting $t=1$, we obtain 
\[
F_{2}(z,1)=G\big(z(1+F(z,1))\big)+G'\big(z(1+F(z,1))\big)zF_{2}(z,1)
\]
and, applying $\frac{\partial}{\partial z}$ to \eqref{FtoG} and
setting $t=1$, we obtain 
\[
F_{1}(z,1)=G'\big(z(1+F(z,1))\big)\Big(1+F(z,1)+zF_{1}(z,1)\Big).
\]
We now eliminate $G'\big(z(1+F(z,1))\big)$ between these two equations,
and use $G\big(z(1+F(z,1))\big)=F(z,1)$ (obtained by setting $t=1$
in \eqref{FtoG}), to give 
\[
\Big(F_{2}(z,1)-F(z,1)\Big)\Big(1+F(z,1)+zF_{1}(z,1)\Big)=F_{1}(z,1)\,zF_{2}(z,1).
\]
The result follows immediately by solving this linear equation for
$F_{2}(z,1)$. 
\end{proof}
\begin{rem}
\label{rem:82} In this remark we discuss an interpretation that the
preceding proposition has, when placed in the framework of free probability.

Let $\mu$ be a ``probability distribution with finite moments of
all orders'', but construed in a purely algebraic sense, when it
is simply a linear functional $\mu\colon\mathbb{C}[X]\to\mathbb{C}$
such that $\mu(1)=1$. We use the notation 
\[
M_{\mu}(z):=\sum_{n=1}^{\infty}\mu(X^{n})z^{n}\ \ \mbox{ (the {\em moment series} of \ensuremath{\mu}).}
\]
In free probability one also associates with $\mu$ another power
series, denoted by $R_{\mu}$ and called the \emph{$R$-transform
of $\mu$}, which in many respects is the free probability analog
for the concept of ``characteristic function of $\mu$'' used in
classical probability. The two series $M_{\mu}$ and $R_{\mu}$ satisfy
the functional equation 
\begin{equation}
R_{\mu}\bigl(\,z(1+M_{\mu}(z))\,\bigr)=M_{\mu}(z),\label{eqn:82b}
\end{equation}
which can in fact be used as the definition of $R_{\mu}$. That is,
$R_{\mu}(z)$ is the unique power series of the form $\sum_{n=1}^{\infty}\alpha_{n}z^{n}$
which satisfies Equation \eqref{eqn:82b} (for details, see e.g. Lecture
16 of \cite{nica-speicher2006lectures}).

Now, an important concept of free probability is the operation of
\emph{free additive convolution} $\boxplus$; this is the operation
with probability distributions which corresponds to the operation
of adding freely independent elements of a non-commutative probability
space (for details, see e.g. Lecture 12 of \cite{nica-speicher2006lectures}).
What is of interest for us here is that the probability distribution
$\mu$ from the preceding paragraph defines a convolution semigroup
$(\mu_{t})_{t\in(0,\infty)}$ with respect to the operation $\boxplus$.
The notation commonly used for $\mu_{t}$ is ``$\mu^{\boxplus t}$'',
and in the case when $t$ is an integer one really has $\mu_{t}=\mu\boxplus\cdots\boxplus\mu$
with $t$ terms in the $\boxplus$-sum. The definition of $\mu_{t}$
for arbitrary (not necessarily integer) $t\in(0,\infty)$ is made
via the equation 
\begin{equation}
R_{\mu_{t}}(z)=t\,R_{\mu}(z);\label{eqn:82c}
\end{equation}
that is, $\mu_{t}$ is the uniquely determined distribution whose
$R$-transform equals the right-hand side of \eqref{eqn:82c}.

Let us next look at the moment series $M_{\mu_{t}}(z)$ of the distributions
$\mu_{t}=\mu^{\boxplus t}$, and let us bundle all these moment series
in one series $F$ of two variables: 
\[
F(z,t):=M_{\mu_{t}}(z).
\]
The functional equation of the $R$-transform (stated as Equation
\eqref{eqn:82b} above) then boils down precisely to the Equation
(\ref{FtoG}) used in the proof of Proposition \ref{pr:101}, where
one puts $G:=R_{\mu}$. So then, Proposition \ref{pr:101} can be
seen as stating the fact that the ``time-derivative'' of $M_{\mu_{t}}(z)$
at time $t=1$ can be written in terms of $M_{\mu}$ and $M_{\mu}'$.
Indeed, as the reader can easily check, the formula in the conclusion
of Proposition \ref{pr:101} takes now the form 
\begin{equation}
\frac{\partial F(z,t)}{\partial t}\mid_{t=1}=M_{\mu}(z)+\frac{zM_{\mu}(z)M_{\mu}'(z)}{1+M_{\mu}(z)}.\label{eqn:82d}
\end{equation}
Note that $M_{\mu}(z)$ and its derivative $M_{\mu}'(z)$ are series
which are calculated at the exact time $t=1$. The point in Equation
\eqref{eqn:82d} is that we can calculate $\frac{\partial F}{\partial t}$
at $t=1$ just from information available at the exact time $t=1$:
there is no need to look at other times $t$ near $1$! 
\end{rem}

We now return to the framework of Proposition \ref{pr:101}. In order
to place the generating function $\Psi^{(i)}(z,t)$ in this framework,
we need to prove that it satisfies condition \eqref{ftog}, which
we do in the next lemma for the three cases $i=1,2,3$.

Some notational details used in the proof of Lemma \ref{le:1015}:
for a set $\alpha$ of positive integers, let $\nc\left(\alpha\right)$
denote the set of non-crossing partitions of the elements of $\alpha$,
ordered from smallest to largest. This is a lattice that is isomorphic
to $\NC(n)$, where $n=\vert\alpha\vert$. The maximum element of
this lattice is the partition of $\alpha$ with a single block, namely
the set $\alpha$ itself, and we denote this maximum element by $1_{\alpha}$. 
\begin{lem}
\label{le:1015} The generating function $\Psi^{(i)}(z,t)$ satisfies
condition \eqref{ftog} for $i=1,2,3$. 
\end{lem}

\begin{proof}
For $i=1$, we introduce some specialized notation for pairs of positive
integers. For a positive integer $i$, let $P_{i}=\left\{ 2i-1,2i\right\} $.
For a set $\alpha$ of positive integers, let $P(\alpha)=\cup_{i\in\alpha}P_{i}$,
and let $\lambda\left(\alpha\right)$ denote the partition of $P\left(\alpha\right)$
in which the blocks are the pairs $P_{i}$ for $i\in\alpha$. In other
words, $\la(\alpha)$ is the interval partition of the even set $P(\alpha)$
in which the blocks are consecutive pairs of elements.

Using this notation, note that for every $m$ and $\rho\in\nc\left(2m\right)$,
we must have $\lambda_{2,m}\vee\rho\ge\lambda_{2,m}$ and so $\lambda_{2,m}\vee\rho$
must have blocks of the form $P(\theta_{1}),P(\theta_{2}),\ldots$,
where $\theta_{1},\theta_{2},\ldots$ are the blocks of some partition
$\theta\in\nc(n)$. Thus, 
\[
\Psi^{(1)}(z,t)=\sum_{m\ge1}z^{m}\sum_{\theta\in\nc\left(m\right)}\;\;\prod_{\theta_{i}\in\theta}\left(t\!\!\sum_{\substack{\tau\in\nc\left(P\left(\theta_{i}\right)\right)\\
\tau\vee\lambda\left(\theta_{i}\right)=1_{P\left(\theta_{i}\right)}
}
}\!\!\!\!\!\!1~~\right),
\]
where the inner sum depends only on the size of the set $\theta_{i}$.
We conclude that $\Psi^{(1)}(z,t)$ is of the form $F\left(z,t\right)$
as in \eqref{ftog} with 
\[
g_{k}=\left|\left\{ \tau\in\nc\left(2k\right)\,\middle|\,\tau\vee\lambda_{2,k}=1_{2k}\right\} \right|,
\]
proving the result for $i=1$.\vspace{0.05in}

For $i=2$, we introduce further notation for interval partitions.
Each interval partition with $k$ blocks is specified by a $k$-tuple
of positive integers ${\mathbf{a}}_{k}=(a_{1},\ldots,a_{k})$, where
$a_{1},\ldots,a_{k}\ge1$ specify the sizes of the blocks in order.
For a positive integer $i$, let $B_{i}^{\left({\mathbf{a}}_{k}\right)}=\left\{ a_{1}+\ldots+a_{i-1}+1,\ldots,a_{1}+\ldots+a_{i}\right\} $,
which is the $i$th block in the interval partition. For a set $\alpha$
of positive integers, let $B^{\left({\mathbf{a}}_{k}\right)}(\alpha)=\cup_{i\in\alpha}B_{i}^{\left({\mathbf{a}}_{k}\right)}$,
and let $\gamma^{\left({\mathbf{a}}_{k}\right)}(\alpha)$ denote the
partition of $B^{\left({\mathbf{a}}_{k}\right)}(\alpha)$ in which
the blocks are the intervals $B_{i}^{\left({\mathbf{a}}_{k}\right)}$
for $i\in\alpha$. We let $\gamma_{k}^{\left({\mathbf{a}}_{k}\right)}=\gamma^{\left({\mathbf{a}}_{k}\right)}\left(\left[k\right]\right)$,
the interval partition with block sizes $a_{1},\ldots,a_{k}$, in
order.

Using this notation, note that for any $a_{1},\ldots,a_{k}\ge1$ and
$\rho\in\nc\left(a_{1}+\ldots+a_{k}\right)$, we must have $\gamma_{k}^{\left({\mathbf{a}}_{k}\right)}\vee\rho\ge\gamma_{k}^{\left({\mathbf{a}}_{k}\right)}$,
and so $\gamma_{k}^{\left({\mathbf{a}}_{k}\right)}\vee\rho$ must
have blocks of the form $B^{\left({\mathbf{a}}_{k}\right)}\left(\theta_{1}\right),B^{\left({\mathbf{a}}_{k}\right)}\left(\theta_{2}\right),\ldots$
where $\theta_{1},\theta_{2},\ldots$ are the blocks of some partition
$\theta\in\nc\left(k\right)$. Thus we have 
\begin{eqnarray}
\Psi^{(2)}(z,t) & = & \sum_{k\ge1}z^{k}\sum_{\substack{\theta\in\nc\left(k\right)\\
a_{1},\ldots,a_{k}\ge1
}
}\;\;\prod_{\theta_{i}\in\theta}\left(t\!\!\!\!\sum_{\substack{\tau\in\nc\left(B^{\left({\mathbf{a}}_{k}\right)}(\theta_{i})\right)\\
\gamma^{\left({\mathbf{a}}_{k}\right)}(\theta_{i})\vee\tau=1_{B^{\left({\mathbf{a}}_{k}\right)}(\theta_{i})}
}
}\!\!\!\!\prod_{j\in\theta_{i}}y^{a_{j}}\right)\nonumber \\
 & = & \sum_{k\ge1}z^{k}\sum_{\substack{\theta\in\nc\left(k\right)}
}\;\;\prod_{\theta_{i}\in\theta}\left(t\!\!\!\!~~\sum_{a_{1},\ldots,a_{\left|\theta_{i}\right|}\ge1}~~\sum_{\substack{\tau\in\nc\left(a_{1}+\ldots+a_{\left|\theta_{i}\right|}\right)\\
\gamma_{m}^{\left({\mathbf{a}}_{m}\right)}\vee\tau=1_{a_{1}+\ldots a_{\left|\theta_{i}\right|}}
}
}\!\!\!\!~~\prod_{j=1}^{\left|\theta_{i}\right|}y^{a_{j}}\right),\label{eq:expression for phi 2}
\end{eqnarray}
where the inner sum in \eqref{eq:expression for phi 2} depends only
on the size of the set $\theta_{i}$. We conclude that $\Psi^{(2)}(z,t)$
is of the form \eqref{ftog} with 
\[
g_{m}=\!\!\!\!\!\!\!\!\sum_{\substack{a_{1},\ldots,a_{m}\ge1\\
\tau\in\nc\left(a_{1}+\ldots+a_{m}\right)\\
\gamma_{m}^{\left({\mathbf{a}}_{m}\right)}\vee\tau=1_{a_{1}+\ldots+a_{m}}
}
}\!\!\!\!\!\!\!\!y^{a_{1}+\ldots+a_{m}},\qquad m\ge1,
\]
proving the result for $i=2$. (We repeat that in this case, $g_{m}$
is a formal power series in the variable $y$, not simply a scalar;
nonetheless, it is independent of the variables $z$ and $t$, which
is all we need in order to apply Proposition \ref{pr:101}.)\vspace{0.05in}

Finally, for $i=3$ we have 
\[
\Psi^{(3)}(z,t)=\sum_{n\ge1}z^{n}\sum_{\theta\in\NC(n)}\;\;\prod_{\theta_{i}\in\theta}\Big(t\sum_{\substack{\tau,\sigma\in\nc\left(\theta_{i}\right)\\
\tau\vee\sigma=1_{\theta_{i}}
}
}\!\!\!\!1\;\;\Big),
\]
where the inner sum depends only on the size of the set $\theta_{i}$.
We conclude that $\Psi^{(3)}(z,t)$ is of the form \eqref{ftog} with
\[
g_{m}=\sum_{\substack{\tau,\sigma\in\nc\left(m\right)\\
\tau\vee\sigma=1_{m}
}
}\!\!\!\!1~=\ \left|\left\{ \pi,\rho\in\nc\left(m\right)\,\middle|\,\pi\vee\rho=1_{m}\right\} \right|,\qquad m\ge1,
\]
proving the result for $i=3$. 
\end{proof}
\begin{rem}
\label{rem:84} Each of the three verifications made in the proof
of Lemma \ref{le:1015} has (on the lines of Remark \ref{rem:82})
an interpretation in terms of a $\boxplus$-convolution semigroup
$(\mu_{t})_{t\in(0,\infty)}$. In each of the three cases, the relevant
probability distribution $\mu$ turns out to be related to the Marchenko-Pastur
distribution, which is the counterpart of the Poisson distribution
in free probability. For illustration, we discuss below in a bit more
detail the first of the three cases, concerning the series $\Psi^{(1)}(z,t)$.

The standard Marchenko-Pastur (or free Poisson) distribution is the
absolutely continuous distribution on $[0,4]$ with density $(2\pi)^{-1}\sqrt{(4-t)/t}dt$.
In the algebraic setting of Remark \ref{rem:82}, where a compactly
supported probability distribution is treated as a linear functional
on $\bC[X]$, the standard free Poisson distribution is the linear
functional $\nu:\bC[X]\to\bC$ uniquely determined by the requirement
that 
\begin{equation}
\nu(1)=1\ \mbox{ and }\ \nu(X^{n})=\Cat_{n},\ \ \forall\,n\in\bN.\label{eqn:84a}
\end{equation}
The above formula for moments translates into a very simple formula
for the $R$-transform of $\nu$: 
\begin{equation}
R_{\nu}(z)=\sum_{n=1}^{\infty}z^{n}=z/(1-z)\label{eqn:84b}
\end{equation}
(see e.g. \cite{nica-speicher2006lectures}, pages 203-206 in Lecture
12).

For our discussion in the present remark, it is convenient to also
consider the framework of a non-commutative probability space $(\cA,\varphi)$
-- this simply means that $\cA$ is a unital algebra over $\bC$
and $\varphi:\cA\to\bC$ is a linear functional normalized such that
$\varphi(1_{_{\cA}})=1$. For an element $a\in\cA$, the linear functional
$\mu_{a}:\bC[X]\to\bC$ determined by the requirement that 
\begin{equation}
\mu_{a}(1)=1\ \mbox{ and }\ \mu_{a}(X^{n})=\varphi(a^{n}),\ \ \forall\,n\in\bN\label{eqn:84c}
\end{equation}
is called the {\em distribution of $a$} with respect to $\varphi$.
An element $a\in\cA$ is said to be a {\em standard free Poisson
element} when its distribution $\mu_{a}$ is the functional $\nu$
from Equation (\ref{eqn:84a}).

Let $a$ be a standard free Poisson element in a non-commutative probability
space $(\cA,\varphi)$, let us fix a positive integer $\ell$, and
let us consider the element $a^{\ell}\in\cA$. Let $\mu$ denote the
distribution of $a^{\ell}$ with respect to $\varphi$. A basic formula
in the combinatorial theory of the $R$-transform (known as the ``formula
for free cumulants with products as arguments'' -- see \cite{nica-speicher2006lectures},
pages 178-181 in Lecture 11) expresses the coefficients of the $R$-transform
of $\mu$ ( = distribution of $a^{\ell}$) in terms of the coefficients
of the $R$-transform of $\nu$ ( = distribution of $a$). This formula
says that for every $m\in\bN$ we have 
\begin{equation}
[z^{m}]\bigl(\,R_{\mu}(z)\,\bigr)=\sum_{\begin{array}{c}
{\scriptstyle {\pi\in\nc\left(\ell m\right)\ such}}\\
{\scriptstyle {that\ \pi\vee\lambda_{\ell,m}=1_{\ell m}}}
\end{array}}\ \Bigl(\,\prod_{V\in\pi}[z^{|V|}]\bigl(\,R_{\nu}(z)\,\bigr)\,\Bigr),\label{eqn:84d}
\end{equation}
where $\lambda_{\ell,m}=\left\{ \left\{ 1,\ldots,\ell\right\} ,\left\{ \ell+1,\ldots,2\ell\right\} ,\ldots,\left\{ \left(m-1\right)\ell+1,\ldots,m\ell\right\} \right\} \in\nc\left(\ell m\right)$
as in Proposition \ref{prop:16}. Since in the case at hand all the
coefficients of $R_{\nu}$ are equal to $1$, Equation (\ref{eqn:84d})
amounts to 
\begin{equation}
[z^{m}]\bigl(\,R_{\mu}(z)\,\bigr)=\ \vline\ \{\pi\in\nc\left(\ell m\right)\mid\pi\vee\lambda_{\ell,m}=1_{\ell m}\}\ \vline\ .\label{eqn:84e}
\end{equation}
Upon considering the $\boxplus$-convolution semigroup $(\mu_{t})_{t\in(0,\infty)}$
where $\mu_{t}=\mu^{\boxplus t}$, we come to the formula 
\begin{equation}
[z^{m}]\bigl(\,R_{\mu_{t}}(z)\,\bigr)=t\cdot\ \vline\ \left\{ \pi\in\nc\left(\ell m\right)\mid\pi\vee\lambda_{\ell,m}=1_{\ell m}\right\} \ \vline\ ,\label{eqn:84f}
\end{equation}
which holds for every $m\in\bN$ and $t\in(0,\infty)$.

A reader who is familiar with the summation formula that gives the
moments of a distribution in terms of the coefficients of its $R$-transform
(the so-called ``moment$\leftrightarrow$free cumulant'' formula,
see e.g. \cite{nica-speicher2006lectures}, pages 175-177 in Lecture
11) will now recognize that the first two paragraphs in the proof
of Lemma \ref{le:1015} are actually performing a transition from
free cumulants to moments, which starts from Equation (\ref{eqn:84f})
in the special case $\ell=2$ and\footnote{It is easily seen that the first two paragraphs in the proof of Lemma
\ref{le:1015} would in fact work for any $\ell\in\bN$ instead of
$\ell=2$.} arrives to 
\begin{equation}
M_{\mu_{t}}(X^{m})=\sum_{\pi\in NC(\ell m)}t^{\left|\pi\vee\lambda_{\ell,m}\right|},\ \ \forall m\in\bN,\,t\in(0,\infty).\label{eqn:84g}
\end{equation}
Hence, with the specialization $\ell=2$, Equation (\ref{eqn:84g})
amounts to the following statement: for every fixed $t\in(0,\infty)$,
the generating function $\Psi^{(1)}(z,t)$ is precisely the moment
series $M_{\mu_{t}}(z)$. This is the interpretation in terms of $\boxplus$-convolution
for why the series $\Psi^{(1)}(z,t)$ fits in the framework of Proposition
\ref{pr:101} (or of Remark \ref{rem:82}).

We conclude this remark by mentioning that meandric systems also relate
to another direction of research in operator algebras, which studies
planar algebras. This relation is established via the so-called ``Temperley-Lieb
diagrams'', which have been known for a while to be related to meandric
systems (see e.g. Sections 4.3, 4.4 of the survey \cite{di2000folding}),
and which also play an important role in the study of planar algebras.
As pointed out to us by Dimitri Shlyakhtenko, some considerations
about meandric systems with prescribed shape of their top part can
in fact be found in the papers \cite{GJS2010} and \cite{CDS2017},
in the guise of calculations with some planar algebra elements denoted
there as ``$\cup$'' and ``$\Cup$''. When translated into the
framework of the present paper, the calculations involving concatenations
of copies of ``$\cup$'' (e.g. in Lemma 2 of \cite{GJS2010} or
in Proposition 2.14 of \cite{CDS2017}) become calculations about
distances in $NC(n)$ that are measured from the base-points $0_{n}\in NC(n)$.
Likewise calculations done in connection to the element ``$\Cup$''
would translate into calculations about distances measured from the
base-points $\lambda_{2,m}$ of our Proposition \ref{prop:17} (and
should thus relate to the free probability considerations presented
above in this remark). To our best knowledge, the precise result stated
in Proposition \ref{prop:17} is not covered from the Temperley-Lieb
angle of studying meandric systems. 
\end{rem}

$\ $

We now return to the main line of this section. We will show how,
by combining Proposition \ref{pr:101} and Lemma \ref{le:1015}, one
can evaluate the three averages $\widetilde{d}_{2,m}$, $d_{n}$ and
$\widetilde{b}_{n}$. We begin with $\widetilde{d}_{2,m}$, which,
as described above, is the average distance between a non-crossing
partition $\rho\in\NC\left(2m\right)$ and the fixed interval partition
$\la_{2,m}=\left\{ \{1,2\},\dots,\{2m-1,2m\}\right\} $. 
\begin{proof}[Proof of Proposition \ref{prop:17} ]
Recall that we need to prove that for $m\ge1$ 
\[
\left(i\right)~~\widetilde{d}_{2,m}=-\frac{3}{2}+\frac{{2m \choose m}2^{2m-1}}{\Cat_{2m}},
\]
and that 
\[
\left(ii\right)~~\lim_{m\rightarrow\infty}\left(\widetilde{d}_{2,m}-\sqrt{2}m\right)=\frac{7\sqrt{2}}{16}-\frac{3}{2}.
\]
For part $\left(i\right)$, from \eqref{distdef} and \eqref{avgeblocks}
we have 
\begin{equation}
\widetilde{d}_{2,m}=m+\left(m+\frac{1}{2}\right)-\frac{2}{\Cat_{2m}}\big[z^{m}\big]\Psi_{2}^{(1)}(z,1).\label{mu1avge}
\end{equation}
Now by Lemma \ref{le:1015} with $i=1$, we can apply Proposition
\ref{pr:101} and obtain 
\[
\Psi_{2}^{(1)}(z,1)=\Psi^{(1)}(z,1)+z\Psi_{1}^{(1)}(z,1)-\frac{z\Psi_{1}^{(1)}(z,1)}{1+\Psi^{(1)}(z,1)}.
\]
But from \eqref{seriesone} we have 
\[
\Psi^{(1)}(z,1)=\sum_{m\ge1}\Cat_{2m}z^{m},\qquad z\Psi_{1}^{(1)}(z,1)=\sum_{m\ge1}m\Cat_{2m}z^{m},
\]
and thus from \eqref{mu1avge} we obtain 
\begin{eqnarray*}
\widetilde{d}_{2,m} & = & 2m+\frac{1}{2}-2-2m+\frac{2}{\Cat_{2m}}\big[z^{m}\big]\frac{z\Psi_{1}^{(1)}(z,1)}{1+\Psi^{(1)}(z,1)}=-\frac{3}{2}+\frac{2}{\Cat_{2m}}\big[z^{m}\big]\frac{z\Psi_{1}^{(1)}(z,1)}{1+\Psi^{(1)}(z,1)}\\
 & = & -\frac{3}{2}+\frac{2}{\Cat_{2m}}\big[z^{2m}\big]\frac{z^{2}\Psi_{1}^{(1)}(z^{2},1)}{1+\Psi^{(1)}(z^{2},1)}
\end{eqnarray*}
To simplify the final term above, we have the closed form\footnote{It is standard that $C\left(z\right):=\sum\Cat_{n}z^{n}=\frac{1-\sqrt{1-4z}}{2z}$,
and here we need $\frac{1}{2}\left(C\left(z\right)+C\left(-z\right)\right)$.} 
\[
1+\Psi^{(1)}(z^{2},1)=\sum_{m\ge0}\Cat_{2m}z^{2m}=\frac{u-v}{4z},
\]
where $u=\sqrt{1+4z}$, $v=\sqrt{1-4z}$. Thus, using the chain rule,
we obtain 
\[
z^{2}\Psi_{1}^{(1)}(z^{2},1)=\frac{z^{2}}{2z}\frac{\partial}{\partial z}\Psi^{(1)}(z^{2},1)=\frac{z}{2}\Big(-\frac{u-v}{4z^{2}}+\frac{u+v}{2zuv}\Big),
\]
and, combining these expressions and simplifying via $u^{2}-v^{2}=8z$
and $u^{2}+v^{2}=2$, we get\footnote{We use here the well-known equality of generating functions $\sum_{n\ge0}\binom{2n}{n}z^{n}=\frac{1}{\sqrt{1-4z}}$.}
\begin{align*}
\frac{z^{2}\Psi_{1}^{(1)}(z^{2},1)}{1+\Psi^{(1)}(z^{2},1)} & =-\frac{1}{2}+\frac{z(u+v)}{uv(u-v)}=-\frac{1}{2}+\frac{(u^{2}-v^{2})(u+v)}{8uv(u-v)}\\
 & =-\frac{1}{2}+\frac{(u+v)^{2}}{8uv}=-\frac{1}{2}+\frac{2uv+2}{8uv}=-\frac{1}{4}+\frac{1}{4uv}\\
 & =-\frac{1}{4}+\frac{1}{4\sqrt{1-16z^{2}}}=\sum_{m\ge1}{2m \choose m}4^{m-1}z^{2m}.
\end{align*}
Part $\left(i\right)$ of the result follows immediately. \vspace{0.1in}
For part $\left(ii\right)$, the asymptotics of central binomial coefficients
and Catalan numbers are well known (see, e.g., \cite[Pages 381, 383--384]{flajolet2009analytic}):
\begin{equation}
{2m \choose m}=\frac{4^{m}}{\sqrt{\pi m}}\Bigg(1-\frac{1}{8m}+\bigO\Big(\frac{1}{m^{2}}\Big)\Bigg),\qquad\Cat_{n}=\frac{4^{n}}{n\sqrt{\pi n}}\Bigg(1-\frac{9}{8n}+\bigO\Big(\frac{1}{n^{2}}\Big)\Bigg).\label{asymCat}
\end{equation}
Combining these, we obtain 
\begin{align*}
\frac{{2m \choose m}2^{2m-1}}{\Cat_{2m}} & =\frac{4^{2m}}{2\sqrt{\pi m}}\Bigg(\frac{4^{2m}}{2m\sqrt{2\pi m}}\Bigg)^{-1}\Bigg(1+\frac{7}{16m}+\bigO\Big(\frac{1}{m^{2}}\Big)\Bigg)\\
 & =\sqrt{2}\,m+\frac{7\sqrt{2}}{16}+\bigO\Big(\frac{1}{m}\Big),
\end{align*}
and part $\left(ii\right)$ of the result follows immediately. 
\end{proof}
Now we consider $d_{n}$, which, as described above, is the average
distance between an interval partition $\pi\in\Int\left(n\right)$
and a non-crossing partition $\rho\in\nc(n)$. Theorem \ref{thm:avg distance from interval partition}
above states that $\lim_{n\to\infty}\left(d_{n}-\frac{2}{3}n\right)=-\frac{28}{27}$.
Before proving it, we first prove the following result: 
\begin{prop}
\label{prop:expression for d_n}For $n\ge1$, 
\begin{equation}
d_{n}=6n+4+\frac{(-1)^{n}}{2^{n-1}\Cat_{n}}-\frac{3}{2^{n-1}\Cat_{n}}\left[y^{n}\right]\frac{1}{(1+y)\sqrt{1-8y}}.\label{eq:d_n - exact expression}
\end{equation}
\end{prop}

\begin{proof}
From \eqref{distdef} and \eqref{avgeblocks} we have 
\begin{equation}
d_{n}=\tfrac{1}{2}(n+1)+\tfrac{1}{2}(n+1)-\frac{2}{2^{n-1}\Cat_{n}}\big[y^{n}\big]\Psi_{2}^{(2)}(1,1).\label{mu2avge}
\end{equation}
Now by Lemma \ref{le:1015} with $i=2$, we can apply Proposition~\ref{pr:101},
and hence obtain 
\[
\Psi_{2}^{(2)}(1,1)=\Psi^{(2)}(1,1)+\Psi_{1}^{(2)}(1,1)-\frac{\Psi_{1}^{(2)}(1,1)}{1+\Psi^{(2)}(1,1)}.
\]
But from~\eqref{seriesone} we have 
\[
\Psi^{(2)}(1,1)=\sum_{n\ge1}2^{n-1}\Cat_{n}y^{n},\qquad\Psi_{1}^{(2)}(1,1)=\sum_{n\ge1}(n+1)2^{n-2}\Cat_{n}y^{n},
\]
and thus from~\eqref{mu2avge} we obtain 
\begin{align*}
d_{n} & =(n+1)-2-(n+1)+\frac{2}{2^{n-1}\Cat_{n}}\left[y^{n}\right]\frac{\Psi_{1}^{(2)}(1,1)}{1+\Psi^{(2)}(1,1)}\\
 & =-2+\frac{2}{2^{n-1}\Cat_{n}}\left[y^{n}\right]\frac{\Psi_{1}^{(2)}(1,1)}{1+\Psi^{(2)}(1,1)}.
\end{align*}
To simplify the final term above, we have the closed forms 
\[
1+\Psi^{(2)}(1,1)=\frac{1}{2}+\sum_{n\ge0}2^{n-1}\Cat_{n}y^{n}=\frac{1+4y-\sqrt{1-8y}}{8y},
\]
and 
\[
\Psi_{1}^{(2)}(1,1)=\sum_{n\ge1}(n+1)2^{n-2}\Cat_{n}y^{n}=\sum_{n\ge1}2^{n-2}{2n \choose n}y^{n}=\frac{1}{4}\left(\frac{1}{\sqrt{1-8y}}-1\right).
\]
Combining these expressions, we get 
\begin{align*}
\frac{\Psi_{1}^{(2)}(1,1)}{1+\Psi^{(2)}(1,1)} & =\frac{2y\left(\frac{1}{\sqrt{1-8y}}-1\right)}{1+4y-\sqrt{1-8y}}\\
 & =\frac{1}{8(1+y)}\left(\frac{1}{\sqrt{1-8y}}-1\right)\left(1+4y+\sqrt{1-8y}\right)\\
 & =-\frac{y}{2(1+y)}+\frac{3y}{2(1+y)\sqrt{1-8y}}\\
 & =-\frac{y}{2(1+y)}+\frac{3}{2\sqrt{1-8y}}-\frac{3}{2(1+y)\sqrt{1-8y}},
\end{align*}
where for the second equality we have rationalized the denominator,
and for the third and fourth equalities we have routinely simplified.
Equation \eqref{eq:d_n - exact expression} follows immediately. 
\end{proof}
\begin{proof}[Proof of Theorem \ref{thm:avg distance from interval partition} ]
To prove that $\lim_{n\rightarrow\infty}\Big(d_{n}-\frac{2}{3}n\Big)=-\frac{28}{27}$,
we follow the treatment in \cite[Chapter VI]{flajolet2009analytic},
referred to as \emph{singularity analysis}. First, 
\[
\left[y^{n}\right]\frac{1}{\left(1+y\right)\sqrt{1-8y}}=\left[y^{n}\right]\frac{8}{\left(8+8y\right)\sqrt{1-8y}}=8^{n}\cdot\left[y^{n}\right]\frac{8}{\left(8+y\right)\sqrt{1-y}}.
\]
We expand $(8+y)^{-1}$ about $y=1$, in this case via a geometric
series, to obtain 
\[
\frac{8}{8+y}=\frac{\frac{8}{9}}{1-\frac{1}{9}(1-y)}=\frac{8}{9}\sum_{k\ge0}\Big(\frac{1}{9}\Big)^{k}(1-y)^{k},
\]
and hence obtain the expansion 
\[
\frac{8}{(8+y)\sqrt{1-y}}=\frac{8}{9}\left(\frac{1}{\sqrt{1-y}}+\frac{1}{9}\sqrt{1-y}\right)+\frac{8\cdot\left(1-y\right)^{1.5}}{81\cdot\left(8+y\right)}.
\]
But, also from \cite[Theorem VI.1]{flajolet2009analytic}, we have
\[
\big[y^{n}\big](1-y)^{\alpha}=\frac{n^{-\alpha-1}}{\Gamma(-\alpha)}\Bigg(1+\frac{\alpha(\alpha+1)}{2n}+\bigO\Big(\frac{1}{n^{2}}\Big)\Bigg),
\]
where $\Gamma\left(\alpha\right)$ is the \emph{Euler Gamma Function}
defined as $\Gamma\left(\alpha\right)\stackrel{\mathrm{def}}{=}\int_{0}^{\infty}e^{-t}t^{\alpha-1}dt$.
From this we deduce, using $\Gamma\big(\tfrac{1}{2}\big)=\sqrt{\pi}$,
$\Gamma\big(-\tfrac{1}{2}\big)=-2\sqrt{\pi}$, that 
\begin{eqnarray*}
\left[y^{n}\right]\frac{1}{\sqrt{1-y}} & = & \frac{1}{\sqrt{\pi n}}\left(1-\frac{1}{8n}+\bigO\left(\frac{1}{n^{2}}\right)\right)\\
\left[y^{n}\right]\sqrt{1-y} & = & \frac{1}{-2\sqrt{\pi n}}\left(\frac{1}{n}+\bigO\left(\frac{1}{n^{2}}\right)\right).
\end{eqnarray*}
Finally, since the function $\frac{\left(1-y\right)^{1.5}}{8+y}$
is $\bigO\left(1-y\right)^{1.5}$ in a neighborhood of $y=1$ and
satisfies the other assumptions of \cite[Theorem VI.3]{flajolet2009analytic},
we have, by this theorem, that 
\[
\left[y^{n}\right]\frac{\left(1-y\right)^{1.5}}{8+y}=\bigO\left(n^{-2.5}\right).
\]
We combine all this to obtain 
\begin{eqnarray*}
\left[y^{n}\right]\frac{1}{\left(1+y\right)\sqrt{1-8y}} & = & 8^{n}\cdot\left[y^{n}\right]\left(\frac{8}{9}\frac{1}{\sqrt{1-y}}+\frac{8}{81}\sqrt{1-y}+\frac{8\cdot\left(1-y\right)^{1.5}}{81\cdot\left(8+y\right)}\right)\\
 & = & \frac{8}{9}\cdot\frac{8^{n}}{\sqrt{\pi n}}\left(\left(1-\frac{1}{8n}+\bigO\left(\frac{1}{n^{2}}\right)\right)+\left(-\frac{1}{18n}+\bigO\left(\frac{1}{n^{2}}\right)\right)+\bigO\left(\frac{1}{n^{2}}\right)\right)\\
 & = & \frac{8}{9}\cdot\frac{8^{n}}{\sqrt{\pi n}}\left[1-\frac{13}{72n}+\bigO\left(\frac{1}{n^{2}}\right)\right].
\end{eqnarray*}
Combining these results with \eqref{asymCat} we obtain 
\begin{eqnarray*}
\frac{3}{2^{n-1}\Cat_{n}}\left[y^{n}\right]\frac{1}{\left(1+y\right)\sqrt{1-8y}} & = & \frac{2\cdot3\cdot n\sqrt{\pi n}}{2^{n}\cdot4^{n}}\left[1+\frac{9}{8n}+\bigO\left(\frac{1}{n^{2}}\right)\right]\frac{8\cdot8^{n}}{9\sqrt{\pi n}}\left[1-\frac{13}{72n}+\bigO\left(\frac{1}{n^{2}}\right)\right]\\
 & = & \frac{16}{3}n\left[1+\frac{17}{18n}+\bigO\left(\frac{1}{n^{2}}\right)\right]=\frac{16}{3}n+\frac{136}{27}+\bigO\left(\frac{1}{n}\right),
\end{eqnarray*}
and the result follows immediately using \eqref{eq:d_n - exact expression}. 
\end{proof}
Finally, we consider $\widetilde{b}_{n}$, which, as described above,
is an upper bound on the average distance between two non-crossing
partitions $\pi,\rho\in\nc(n)$. Before proving Proposition \ref{prop:upper bound for average distance}
concerning the asymptotics of $\widetilde{b}_{n}$, we first show: 
\begin{prop}
\label{prop:exact expression for mu_n}For $n\ge1$ we have 
\[
\widetilde{b}_{n}=-n-1+\frac{2n}{\Cat_{n}^{2}}\big[z^{n}\big]\log\left(1+\sum_{k\ge1}\Cat_{k}^{2}z^{k}\right).
\]
\end{prop}

\begin{proof}
From \eqref{distdef} and \eqref{avgeblocks} we have 
\begin{equation}
\widetilde{b}_{n}=\tfrac{1}{2}(n+1)+\tfrac{1}{2}(n+1)-\frac{2}{\Cat_{n}^{2}}\left[z^{n}\right]\Psi_{2}^{(3)}(z,1).\label{mu3avge}
\end{equation}
Now by Lemma~\ref{le:1015} with $i=3$, we can apply Proposition
\ref{pr:101}, and hence obtain 
\[
\Psi_{2}^{(3)}(z,1)=\Psi^{(3)}(z,1)+z\Psi_{1}^{(3)}(z,1)-\frac{z\Psi_{1}^{(3)}(z,1)}{1+\Psi^{(3)}(z,1)}.
\]
But from \eqref{seriesone} we have 
\[
\Psi^{(3)}(z,1)=\sum_{n\ge1}\Cat_{n}^{2}z^{n},\qquad z\Psi_{1}^{(3)}(z,1)=\sum_{n\ge1}n\Cat_{n}^{2}z^{n},
\]
and thus from \eqref{mu3avge} we obtain 
\begin{align*}
\widetilde{b}_{n} & =(n+1)-2-2n+\frac{2}{\Cat_{n}^{2}}\big[z^{n}\big]\frac{z\Psi_{1}^{(3)}(z,1)}{1+\Psi^{(3)}(z,1)}\\
 & =-n-1+\frac{2}{\Cat_{n}^{2}}\big[z^{n}\big]\frac{z\Psi_{1}^{(3)}(z,1)}{1+\Psi^{(3)}(z,1)}\\
 & =-n-1+\frac{2}{\Cat_{n}^{2}}\big[z^{n}\big]z\frac{\partial}{\partial z}\log\big(1+\Psi^{(3)}(z,1)\big),
\end{align*}
and the result follows from the fact that $\big[z^{n}\big]z\frac{\partial}{\partial z}f(z)=n\big[z^{n}\big]f(z)$
for any formal power series $f$. 
\end{proof}
\begin{proof}[Proof of Proposition \ref{prop:upper bound for average distance}]
We need to show that 
\[
\lim_{n\rightarrow\infty}\frac{\widetilde{b}_{n}}{n}=\frac{3\pi-8}{8-2\pi}.
\]
Note that the series 
\[
\Phi(z)=1+\Psi^{(3)}(z,1)
\]
has radius of convergence $\tfrac{1}{16}$ -- this follows, e.g.,
from the asymptotics of Catalan numbers in \eqref{asymCat}. Then,
following the treatment in \cite[Section VI.10.2]{flajolet2009analytic}
of Hadamard products of series as part of closure properties in singularity
analysis, we obtain 
\[
\lim_{n\rightarrow\infty}\frac{1}{\Cat_{n}^{2}}\big[z^{n}\big]\log\Phi(z)=\frac{1}{\Phi\big(\tfrac{1}{16}\big)}.
\]
But from \cite[Page 132]{Lando1992meanders} we have 
\[
\Phi\big(\tfrac{1}{16}\big)=\frac{4(4-\pi)}{\pi},
\]
and the result follows immediately. 
\end{proof}

\section*{Acknowledgements}

We thank Bruce Richmond for extensive help with the asymptotics in
Section \ref{sec:Distance-distributions}, and the anonymous referee
for drawing our attention to the shorter proof of Theorem \ref{thm:main-shallow meanders}
described in Remark \ref{rem:shorter-proof-for-thm-1.1}. A. N. would
also like to thank Dimitri Shlyakhtenko for pointing out how meandric
systems relate to the research literature on planar algebras.

\bibliographystyle{amsalpha}
\bibliography{GNP}

\providecommand{\bysame}{\leavevmode\hbox to3em{\hrulefill}\thinspace}
\providecommand{\MR}{\relax\ifhmode\unskip\space\fi MR }
\providecommand{\MRhref}[2]{%
  \href{http://www.ams.org/mathscinet-getitem?mr=#1}{#2}
}
\providecommand{\href}[2]{#2}
\begin{thebibliography}{DFGG97}

\bibitem[AP05]{albert2005bounds}
M.H. Albert and M.S. Paterson, \emph{Bounds for the growth rate of meander
  numbers}, J. Combin. Theory Ser. A \textbf{112} (2005), no.~2, 250--262.

\bibitem[Arn88]{arnol1988branched}
V.I. Arnol'd, \emph{A branched covering of {CP}$^{\;2}\to${S}$^{\;4}$,
  hyperbolicity and projectivity topology}, Sib. Math. J. \textbf{29} (1988),
  no.~5, 717--726.

\bibitem[Bia97]{biane1997some}
Ph. Biane, \emph{Some properties of crossings and partitions}, Discrete Math.
  \textbf{175} (1997), no.~1-3, 41--53.

\bibitem[CDS17]{CDS2017}
S.~Curran, Y.~Dabrowski, and D.~Shlyakhtenko, \emph{Free analysis and planar
  algebras}, Proceedings of the 2014 Maui and 2015 Qinhuangdao Conferences in
  Honour of Vaughan FR Jones' 60th Birthday, vol.~46, Centre for Mathematics
  and its Applications, Mathematical Sciences Institute, The Australian
  National University, 2017, Also available as arXiv:1411.0268., pp.~115--142.

\bibitem[DF00]{di2000folding}
P.~Di~Francesco, \emph{Folding and coloring problems in mathematics and
  physics}, Bull. Amer. Math. Soc. (N.S.) \textbf{37} (2000), no.~3, 251--307.

\bibitem[DFGG96]{di1996meanders}
P.~Di~Francesco, O.~Golinelli, and E.~Guitter, \emph{Meanders: a direct
  enumeration approach}, Nuclear Phys. B \textbf{482} (1996), no.~3, 497--535.

\bibitem[DFGG97]{di1997meander}
\bysame, \emph{Meander, folding, and arch statistics}, Math. Comp. Model.
  \textbf{26} (1997), no.~8-10, 97--147.

\bibitem[Fra98]{franz1998partial}
R.O.W. Franz, \emph{A partial order for the set of meanders}, Ann. Comb.
  \textbf{2} (1998), no.~1, 7--18.

\bibitem[FS09]{flajolet2009analytic}
Ph. Flajolet and R.~Sedgewick, \emph{Analytic combinatorics}, Cambridge
  University press, 2009.

\bibitem[GJ83]{goulden1983combinatorial}
I.P. Goulden and D.M. Jackson, \emph{Combinatorial enumeration}, John Wiley \&
  Sons, New York, 1983.

\bibitem[GJS10]{GJS2010}
A.~Guionnet, V.F.R. Jones, and D.~Shlyakhtenko, \emph{Random matrices, free
  probability, planar algebras and subfactors}, Quanta of maths, Clay
  Mathematics Proceedings, vol.~11, 2010, Also available as arXiv:0712.2904.,
  pp.~201--239.

\bibitem[Hal06]{hall2006meanders}
H.T. Hall, \emph{Meanders in a {C}ayley graph}, arXiv preprint math/0606170,
  2006.

\bibitem[JG00]{jensen2000critical}
I.~Jensen and A.J. Guttmann, \emph{Critical exponents of plane meanders}, J.
  Phys. A \textbf{33} (2000), no.~21, L187--L192.

\bibitem[Kre72]{kreweras1972partitions}
G.~Kreweras, \emph{Sur les partitions non crois{\'e}es d'un cycle}, Discrete
  Math. \textbf{1} (1972), no.~4, 333--350.

\bibitem[LC03]{lacroix2003approaches}
M.~La~Croix, \emph{Approaches to the enumerative theory of meanders}, Master's
  thesis, University of Waterloo, 2003, available at
  http://math.mit.edu/~malacroi/Latex/Meanders.pdf.

\bibitem[LZ92]{Lando1992meanders}
S.K. Lando and A.K. Zvonkin, \emph{Meanders}, Selecta Math. Sovietica
  \textbf{11} (1992), no.~2, 117--144.

\bibitem[LZ04]{lando2004graphs}
\bysame, \emph{Graphs on surfaces and their applications}, Encyclopaedia of
  Mathematical Sciences, vol. 141, Springer-Verlag, 2004.

\bibitem[NS06]{nica-speicher2006lectures}
A.~Nica and R.~Speicher, \emph{Lectures on the combinatorics of free
  probability}, vol.~13, Cambridge University Press, 2006.

\bibitem[Sav09]{savitt2009polynomials}
D.~Savitt, \emph{Polynomials, meanders, and paths in the lattice of noncrossing
  partitions}, Trans. Amer. Math. Soc. \textbf{361} (2009), no.~6, 3083--3107,
  also available at arXiv:math/0606169.

\bibitem[Spe97]{speicher1997}
R.~Speicher, \emph{On universal products}, Fields Inst. Commun. (1997), no.~12,
  257--266.

\bibitem[Sta12]{stanley2011enumerative}
R.P. Stanley, \emph{Enumerative combinatorics}, 2 ed., vol.~1, Cambridge
  University Press, 2012.

\bibitem[VP]{summer-project}
E.~Verdugo-Paredes, Undergraduate research project at the University of
  Waterloo, under supervision of K.G. Hare and A. Nica, Spring Term 2016.

\end{thebibliography}

\noindent I.~P.~Goulden, Department of Combinatorics and Optimization,
University of Waterloo, Ontario, Canada. ipgoulden@uwaterloo.ca\\

\noindent Alexandru Nica, Department of Pure Mathematics, University
of Waterloo, Ontario, Canada. anica@uwaterloo.ca\\

\noindent Doron Puder, School of Mathematical Sciences, Tel Aviv University,
Tel Aviv, 6997801, Israel. doronpuder@gmail.com 
\end{document}